 \documentclass[12pt]{article}
 \usepackage{amsmath}

\usepackage{amssymb,amsthm}
 \usepackage{mathrsfs,amsfonts}
 \usepackage{color}
 \usepackage{graphicx}
\usepackage{bbm}
\usepackage{ulem}

\newcommand{\ddr}{\mathrm{d}}
 \newtheorem{theorem}{Theorem}[section]
 \newtheorem{lemma}[theorem]{Lemma}
 \newtheorem{corol}[theorem]{Corollary}
 \newtheorem{prop}[theorem]{Proposition}
 \newtheorem{example}[theorem]{Example}
 \newtheorem{Definition}[theorem]{Definition}
 \newtheorem{con1}[theorem]{Condition}
\theoremstyle{remark}
 \newtheorem{remark}[theorem]{Remark}

 \def\blemma{\begin{lemma}\sl{}\def\elemma{\end{lemma}}}

 \def\bcorollary{\begin{corol}\sl{}\def\ecorollary{\end{corol}}}

 \def\beqlb{\begin{eqnarray}}\def\eeqlb{\end{eqnarray}}
 \def\beqnn{\begin{eqnarray*}}\def\eeqnn{\end{eqnarray*}}

\theoremstyle{remark}

\def\bcorollary{\begin{corol}\sl{}\def\ecorollary{\end{corol}}}

\def\beqlb{\begin{eqnarray}}\def\eeqlb{\end{eqnarray}}
\def\beqnn{\begin{eqnarray*}}\def\eeqnn{\end{eqnarray*}}

\def\<{\langle}\def\>{\rangle}

\def\ep{\varepsilon}

\def\mbb{\mathbb}

\def\ar{\!\!&}

\begin{document}

		\centerline{\Large Time-changed spectrally positive L\'evy processes}
		 \bigskip
		\centerline{\Large started from infinity}

		\bigskip\bigskip\bigskip
		
		\centerline{\large  Cl\'ement Foucart \footnote{Universit\'e Sorbonne Paris Nord and Paris 8, Laboratoire Analyse, G\'eom\'etrie $\&$ Applications, UMR 7539. Institut Galil\'ee, 99 avenue J.B. Cl\'ement, 93430 Villetaneuse, France, foucart@math.univ-paris13.fr}, Pei-Sen Li \footnote{Institute for Mathematical Sciences, Renmin University of China, 100872 Beijing, P. R. China, peisenli@ruc.edu.cn} and Xiaowen Zhou \footnote{Department of Mathematics and Statistics, Concordia University, 1455 De Maisonneuve Blvd. W., Montreal, Canada, xiaowen.zhou@concordia.ca}}
			

		\begin{abstract}
			Consider a spectrally positive L\'evy process $Z$ with log-Laplace exponent $\Psi$  and a positive continuous function $R$ on $(0,\infty)$. We investigate the entrance from $\infty$ of the process $X$ obtained by changing time in $Z$ with the inverse of the additive functional $\eta(t)=\int_{0}^{t}\frac{\ddr s}{R(Z_s)}$. We provide a necessary  and sufficient condition for infinity to be an entrance boundary of the process $X$. Under this condition, the process can start from infinity and we study its speed of coming down from infinity. When the L\'evy process has a negative drift $\delta:=-\gamma<0$, sufficient conditions over $R$ and $\Psi$ are found for the process to come down from infinity along the deterministic function $(x_t,t\geq 0)$ solution to $\ddr x_t=-\gamma R(x_t)\ddr t$, with $x_0=\infty$.
			When $\Psi(\lambda)\sim c\lambda^{\alpha}$, with $\lambda \rightarrow 0$,  $\alpha\in (1,2]$, $c>0$ and $R$ is regularly varying at $\infty$ with index $\theta>\alpha$, the process comes down from infinity and we find a renormalisation in law of its running infimum at small times.
		\end{abstract}
		
		\noindent\textbf{MSC (2010):} primary 60J80; secondary 60H30, 92D15, 92D25.
		
		\noindent \textbf{Keywords.}
			Coming down from infinity; entrance boundary; hitting times; non-linear branching processes; regularly varying functions; spectrally positive L\'evy processes; time-change; weighted occupation times

	\section{Introduction}
	\setcounter{equation}{0}
	\indent{}
	A classical question in the theoretical study of Markov processes is to know if a given process can be started from its boundaries. In the framework of birth-death processes and diffusions, Feller's tests provide necessary and sufficient conditions for a boundary to be an entrance. Namely, the boundary cannot be reached but the process can start from it, see e.g. Anderson \cite[Chapter 8]{MR1118840} and Karlin and Taylor \cite[Chapter 15]{MR611513} for a comprehensive account on Feller's tests.  No explicit tests classifying boundaries are available for general Markov processes, and taylor-made criteria have to be designed for  a given class of processes. We mention for instance the recent article of D\"oring and Kyprianou \cite{2018arXiv180201672D}, where integral tests classifying the boundary $\pm{\infty}$ of diffusions with stable jumps are found.

	We consider in this article  the class of Markov processes obtained as time-changes of L\'evy processes with no negative jumps.
	We will provide a simple test for such processes to have infinity as entrance boundary and investigate further their small-time asymptotics when they leave $\infty$. Such studies have been carried out for instance for birth-death processes, see Bansaye et al. \cite{MR3595771}, Sagitov and France \cite{MR3707825}, and for Kolmogorov diffusions, see Bansaye et al. \cite{2017arXiv171108603B}.

	Time-changed spectrally positive L\'evy processes have been considered from the point of view of stochastic population models,  see  Li \cite{2016arXiv160909593L}, Li and Zhou \cite{2018arXiv180905759L} and  Li  et al. \cite{2017arXiv170801560L}, where they appear as  solutions to certain stochastic differential equations with jumps. We retain the name coined in those works and call the class of processes under study, \textit{nonlinear continuous-state branching processes} (nonlinear CSBPs for short).
	
	We now give a formal definition. Consider a continuous function $R$ on $[0,\infty)$ strictly positive on $(0,\infty)$, and $Z$ a L\'evy process with no negative jumps started from $x> 0$ with log-Laplace exponent $\Psi$, i.e. $\mbb{E} [e^{-\lambda Z_1}|Z_0=x]=e^{-\lambda x+\Psi(\lambda)}$ for any $\lambda\ge0.$
	We define a nonlinear CSBP with \textit{branching mechanism} $\Psi$ and  \textit{branching rate} $R$ as the process $(X_t,0\leq t <\zeta)$ given by
	\begin{equation}\label{timechangedef}X_t:=Z_{ \eta^{-1}(t)\wedge\tau_0^{-}} \end{equation} where $\tau_0^{-}$ is the first passage time below zero of $Z$ and for any $s\geq 0$,
	\[\eta(s):=\int_{0}^{s\wedge \tau^-_0}\frac{\ddr r}{R(Z_r)} \text{ and } \eta^{-1}(t):=\inf\{s\geq 0: \eta(s)>t\}\]
	with the convention $\inf\emptyset:=\infty$ and $\zeta=\eta(\tau_0^{-})\in (0,\infty]$.
	Since $R$ is strictly positive on $(0, \infty)$, for $0<t<\tau_0^-$, $\frac{1}{R(Z_r)}$ is bounded on $r\in [0, t]$, and hence, the integral $\eta(t)$ is finite for $0<t<\tau_0^-$.
	We also follow the terminology of branching processes by saying that a nonlinear CSBP is {\it supercritical}, {\it critical} or {\it subcritical} if respectively, $\gamma<0$, $\gamma=0$ or $\gamma>0$, where $\gamma=\Psi'(0+)$.
	
	When the function $R$ is the identity, the time-change procedure \eqref{timechangedef} above corresponds to the Lamperti transformation which relates a spectrally positive L\'evy process $Z$ into a CSBP. We refer to e.g. \cite[Theorem 12.2]{MR1406564},   Bingham \cite{MR0410961} and Caballero et al. \cite{MR2243877} for a study of CSBPs by random time-change. It is worth noticing that CSBPs cannot be started from $\infty$ as the boundary $\infty$ is absorbing, see e.g. \cite[Proposition 1]{MR2243877}. Processes obtained by the time-change \eqref{timechangedef} for a general function $R$ are natural generalisations of CSBPs, in which, heuristically the underlying population evolves non linearly at a rate governed by the function $R$.

	When the function $R$ is an exponential function, namely $R(x)=e^{\theta x}$, with $\theta>0$, for all $x\geq 0$,  the second Lamperti's transformation, see e.g. \cite[Theorem 13.1]{MR1406564}, entails that the process $(Y_t,t\geq 0)$ defined by $Y_t:=e^{-X_t}$ for $t\geq 0$, is a positive self-similar Markov process (PSSMP for short) with index of self-similarity $\theta$ and L\'evy parent process $Z$. The property of entrance from $\infty$ for the process $X$ is related to the possibility for $Y$ to be started at $0$. This latter topic has been extensively studied, see Caballero and Chaumont \cite{MR2243877} and the references therein.
	
	We shall mainly focus on subcritical and critical nonlinear CSBPs, as it will  come up, without surprise, that supercritical ones never come down from infinity. Observe also that (sub)-critical nonlinear CSBPs cannot explode (i.e. hit $\infty$ in finite time)  as in these cases,
	$\mbb{P}(\sup_{0\le t<\infty} X_t<\infty)=\mbb{P}(\sup_{t\le \tau^-_0}Z_t<\infty)=\mbb{P}(\tau^-_0<\infty)=1$.

	The following symbolic representation of a sample path of $X$  with $X_0=x$
	illustrates how the time-change \eqref{timechangedef} shrinks time without changing jump sizes and starting levels of the jumps of $Z$.\\
	\begin{figure}[h!]
		\centering \noindent
		\includegraphics[scale=1.4]
		{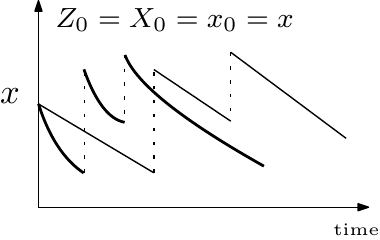}
		\caption{The lighter line represents a L\'evy process $Z$ with drift $-\gamma<0$ and the heavier one the subcritical nonlinear CSBP $X$ obtained by time-change. Both processes start at $x$ and the jump rate of $X$ is greater where $R$ takes large values.}
		\label{behaviornonlinearcb}
	\end{figure}

	If the branching rate function $R$ takes arbitrarily large values on a neighborhood of $\infty$, then  jumps will occur faster and faster as the initial state $x$ goes to infinity. The main problem is thus to see whether those jumps will prevent the process ``started at infinity" to escape or not. The second question is to understand how the process \textit{comes down from infinity} when $\infty$ is an entrance boundary.   We call speed of coming down from infinity, any function $(x_t,t\geq 0)$ such that  $X_t/x_t \underset{t\rightarrow 0}{\longrightarrow} 1$ in probability or almost-surely as $t$ goes to $0+$, when $X_0=\infty$ a.s.
	
	Denote by $-\gamma t$ the drift part of the underlying L\'evy process $Z$, with $\gamma>0$ in the subcritical case. We see heuristically in Figure \ref{behaviornonlinearcb}, that between each jump, the subcritical nonlinear CSBP, follows the deterministic flow
	\begin{equation}\label{ode}
	\ddr x_t=-\gamma R(x_t)\ddr t.
	\end{equation}
	A natural candidate for a speed function in the subcritical case is therefore $(x_t,t\geq 0)$ solution to \eqref{ode} started from $x_0=\infty$. We will find conditions on $\Psi$  and $R$ entailing indeed that $X_t/x_t \underset{t\rightarrow 0}{\longrightarrow} 1$ in probability or almost-surely. In the critical case $\gamma=0$, the nonlinear CSBP oscillates and we find some renormalisation in law of the  running infimum $(\underline{X}_t,t\geq 0)$ defined by $\underbar{X}_t=\inf_{0\le s\le t} X_s$ for specific branching mechanisms and branching rates.
	
	The paper is organized as follows. In Section \ref{preliminaries}, we give a rigorous definition of infinity as entrance boundary and gather fundamental properties of nonlinear CSBPs, such as their Feller property. We recall the definition of the scale function of a spectrally positive L\'evy process as well as some properties of weighted occupation times. We state our main results in Section \ref{results}. In Section \ref{firstentrance}, we study the  Laplace transform and the moments of the first passage times of nonlinear CSBPs started from infinity, below a fixed level. In Section \ref{asymptoticsTb}, we investigate the asymptotics of the latter, when the level tends to $\infty$. In Section \ref{speeds}, we study  the small-time asymptotics of the process  started from infinity.
	\\
	
	\noindent \textbf{Notation.} We use in the sequel Landau's notation: for any positive functions $f$ and $g$, we write $f(z)\sim g(z)$ as $z$ goes to a real number $a$, if $\frac{f(z)}{g(z)}\underset{z\rightarrow a}{\longrightarrow} 1$, $f(z)=O(g(z))$ as $z$ goes to $\infty$ if for a large enough $z_0$, $\sup_{z\geq z_0}\frac{f(z)}{g(z)}<\infty$ and $f(z)=o(g(z))$ as $z$ goes to $\infty$ if $\frac{f(z)}{g(z)}\underset{z\rightarrow \infty}{\longrightarrow} 0$. The integrability of a function $f$ in a neighborhood of $\infty$ (respectively, $0$) is written for short as $\int^{\infty}f(x)\ddr x<\infty$ (respectively, $\int_{0}f(x)\ddr x<\infty$) and $||f||$ denotes the supremum norm of $f$. For any $x\in [0,\infty]$, let $\mbb{P}_x$ be the distribution  of $X$ with $X_0=x$,  let $\mathbb{E}_x$ and $\text{Var}_x$ be the corresponding expectation and variance, respectively. The equality in law is denoted by $\overset{\mathcal{L}}{=}$.

	\section{Preliminaries }\label{preliminaries}
	We provide in this section some fundamental properties of the process $(X_t\geq 0)$ and give a rigorous definition of infinity as entrance boundary. Some results on spectrally positive L\'evy processes that we shall need later are stated.
	\subsection{Nonlinear CSBPs started from infinity}
	Recall that a nonlinear CSBP $X$ satisfies \eqref{timechangedef} with $R$ a continuous and strictly positive function on $(0,\infty)$ and where $\eta$ is the additive functional $\eta:t\mapsto \int_{0}^{t}\frac{\ddr s}{R(Z_s)}$. The process $X$ has obviously no negative jumps and we show in the next proposition that it has the Feller property. Recall that $\tau_0^{-}$ denotes the first passage time below $0$ of the L\'evy process $Z$ and $\zeta=\eta(\tau_{0}^{-})$ is the lifetime of the nonlinear CSBP $X$.
	
	\begin{prop}\label{Fellerproperty}   Assume that $R$ is continuous and strictly positive  on $(0,\infty)$ and that $Z$ is a spectrally positive L\'evy process. For any $x>0$ and any $0\leq t<\tau_0^{-}$, we have $\eta(t)<\infty$ $\mathbb{P}_x$-a.s. The process $(X_t,t\geq 0)$ is well-defined, strong Markov and with c\`adl\`ag paths. Its semigroup $(P_t,t\geq 0)$ satisfies 
		for any function $f$ bounded and continuous on $[0,\infty)$,
		\begin{itemize}
			\item[(i)]  $P_tf$ is continuous on $[0,\infty)$,
			\item[(ii)]
			$P_tf(x)\underset{t\rightarrow 0+}{\longrightarrow} f(x)$ for any $x\in [0,\infty)$.\end{itemize}
	\end{prop} The proof of Proposition \ref{Fellerproperty} is postponed to the Appendix.
	As our primary aim is to study the boundary $\infty$ of the nonlinear CSBP $X$, we  precise now  the definition  of entrance boundary that we adopt. We refer for this definition to Kallenberg \cite[Chapter 23]{Kallenberg}. Set $T_b:=\inf\{t\geq 0: X_t< b\}$ for all $b\geq 0$ with the convention $\inf\emptyset:=\infty$. The random variable $T_b$ is the first passage time of the nonlinear CSBP $X$ below the level $b$ and is a stopping time in the natural filtration of $X$.
	\begin{Definition}
		The boundary $\infty$ is said to be an \textit{instantaneous entrance boundary} for the process $(X_t,t\geq 0)$ if the process does not explode and
		\begin{equation}\label{0.5'} \forall t>0, \underset{b\rightarrow \infty}{\lim}\underset{x\rightarrow \infty}{\liminf}\ \mathbb{P}_x(T_b\leq t)=1.
		\end{equation}
		
	\end{Definition}
	
	The following lemma provides equivalent conditions for $\infty$ to be an entrance for processes with no negative jumps.  We refer to \cite[Lemma 1.2]{noteentrance} for its proof.
	\begin{lemma}\label{equivcdi} Consider a strong Markov process $(X_t,t\geq 0)$ with no negative jumps. The following statements are equivalent:
		\begin{itemize}
			\item[(a)] Condition \eqref{0.5'} holds; 
			\item[(b)] For large enough $b$, $\sup_{x\geq b}\mathbb{E}_x [T_b]<\infty$;
			\item[(c)] $\underset{b\rightarrow \infty}{\lim}\ \underset{x\rightarrow \infty}{ \lim }\mathbb{E}_x [T_b]=0$.
		\end{itemize}
	\end{lemma}
	Condition \eqref{0.5'}, with Proposition \ref{Fellerproperty} (Feller property), ensures the existence of a c\`adl\`ag strong Markov process $X$ on $[0,\infty]$, which starts from infinity and leaves it instantaneously. Namely,  $\mathbb{P}_\infty(X_0=\infty)=1$ and $\mathbb{P}_\infty(X_t<\infty \text{ for any } t>0)=1$. The process has, moreover, the same law as the nonlinear CSBP defined in \eqref{timechangedef} under $\mathbb{P}_x$ for any $x\in (0,\infty)$.   We refer the reader to \cite[Theorem 2.2]{noteentrance}.
	The next proposition, also established in \cite{noteentrance}, is crucial in the study of the speed of coming down from infinity.
	\begin{prop}\label{momentunderPinfty} Let $h$ be a nonnegative continuous and increasing function on $[0,\infty)$. Suppose that \eqref{0.5'} holds. Then
		\begin{itemize}
			\item[(a)]
			for any $\theta>0$, there exists $b_\theta>0$, such that for all $b\geq b_\theta$,
			$\mathbb{E}_\infty(e^{\theta T_b})<\infty;$
			\item[(b)] for any $b>0$, if $T_b<\infty$, $\mathbb{P}_\infty$-almost-surely and $\mathbb{E}_\infty(h(T_b))<\infty$, then \\
			$ \mathbb{E}_x (h(T_b))\underset{x\rightarrow \infty}{\longrightarrow} \mathbb{E}_\infty (h(T_b)).$
		\end{itemize}
		In particular, we see that, under $\mathbb{P}_\infty$, one can find a large enough $b$ such that $T_b$ has  moments of all orders.
	\end{prop}
	
	
	%
	We close this section with the following simple observation.
	\begin{lemma}\label{supercritical} In the supercritical case  the boundary $\infty$ is not an entrance boundary. 
	\end{lemma}
	\begin{proof}
		Assume $\Psi'(0+)=\gamma<0$. Since the L\'evy process $(Z_t,t\geq 0)$ drifts towards $+\infty$,  under $\mathbb{P}_x$, $(Z_t,t\geq 0)$ stays above any level $b<x$ with positive probability. On this event, the time-changed process $(X_t,t\geq 0)$ also stays above $b$. This entails $\mathbb{E}_x(T_b)=\infty$ and  condition \textrm{(b)} in Lemma \ref{equivcdi} is not fulfilled.
	\end{proof}
	
	\subsection{Scale function and weighted occupation times}\label{scalefunction}
	Notice that the dual process $\hat{Z}:=-Z$ is a spectrally negative L\'{e}vy process with Laplace exponent $\Psi$, i.e for any $\lambda\geq 0$,\begin{equation} \label{LK} \Psi(\lambda):=\frac{1}{t}\log \mathbb{E}_0[e^{\lambda \hat{Z}_t}]=\gamma \lambda+\frac{\sigma^2}{2}\lambda^2+\int_{0}^{\infty}\left(e^{-\lambda z}-1+\lambda z\right)\pi(\ddr z)
	\end{equation}
	where $\gamma $ is the drift, $\pi$ is the L\'evy measure satisfying $\int^\infty_0z\wedge z^2\pi(\ddr z)<\infty$ and $\sigma^2/2$ is the Brownian coefficient.
	
	As fluctuations of the L\'evy process $Z$ will play an important role, we recall the definition of the scale function $W$ and some of its basic properties. Some classical results on exit problems for spectrally positive L\'evy processes are summarized in the following lemma. For any $y$, set $\tau^+_y:=\inf\{t\ge0: Z_t> y\}$ and $\tau^-_y:=\inf\{t\ge0: Z_t< y\}.$
	\begin{lemma}\label{l1} Suppose that $\gamma\geq 0$. There exists a strictly increasing continuous function $W$, called scale function, such that $W(x)=0$ for any $x<0$ and for any $q>0$
		\begin{equation}\label{defscale}
		\int_{0}^{\infty}e^{-qy}W(y)\ddr y=\frac{1}{\Psi(q)}.
		\end{equation}
		For any $x\in \mathbb{R}$, $\frac{W(y)}{W(x+y)}\underset{y\rightarrow \infty}{\longrightarrow} 1$, $W(y)\underset{y\rightarrow \infty}{\longrightarrow} W(\infty):=\frac{1}{\gamma}\in ]0,\infty]$, $W(y)=o(e^{py})$ when $y\rightarrow \infty$, for any $p>0$. Moreover, for any $x,y\geq 0$
		\begin{equation*}\label{scaleidentities}\mathbb{P}_0(\tau_x^{+}< \infty)=\frac{W(\infty)-W(x)}{W(\infty)}\text{ and } \mathbb{P}_x(\tau_0^{-}\leq \tau^{+}_{x+y})=\mathbb{P}_0(\tau_{-x}^{-}\leq \tau^{+}_{y})=\frac{W(y)}{W(x+y)}
		\end{equation*}
		and there exist $c_1,c_2>0$ such that for any  $x\ge0$
		\begin{equation}\label{boundW}
		c_1\frac{1}{x\Psi\left(1/x\right)} \leq W(x)\leq c_2\frac{1}{x\Psi\left(1/x\right)}.
		\end{equation}
	\end{lemma}
	We refer to Bertoin \cite[Chapter VII]{MR1406564}  and Kyprianou \cite[Chapter 8]{MR3155252} for proofs. See Propositions III.1 and VII.10 in \cite{MR1406564} for (\ref{boundW}). Scale functions are rarely explicit, and we refer the reader to Kuznetsov et al. \cite{MR3014147}. In the critical stable case, for which $\Psi(\lambda)=\lambda^{\alpha}$ with $\alpha\in (1,2]$, the scale function can be found for instance in \cite[Example 4.17]{MR3014147}, and $W(x)=x^{\alpha-1}/\Gamma(\alpha)$, for $x>0$.

	The following theorem about weighted occupation times  will allow us to study the Laplace transforms of the first entrance times. Its proof follows closely arguments of Li and Palmowski \cite{MR3849809} and is postponed to the appendix.
	
	\begin{theorem}\label{occupation}
		Given a  locally bounded nonnegative function $\omega$ on $(0, \infty)$, let $W^{(\omega)}_n(x)$ satisfy
		\begin{equation}\label{occupation_a}
		W^{(\omega)}_0(x)=1,\quad W^{(\omega)}_{n+1}(x)=\int_{x}^{\infty}W(z-x)\omega(z)W^{(\omega)}_n(z)\ddr z\qquad \text{for} \,\, x\geq 0, n\ge0.
		\end{equation}
		Given $b\geq 0$, if
		$\sum^\infty_{n=0}  W^{(\omega)}_{n}(b)<\infty,$
		then for all $x\geq b$,
		\[ \mbb{E}_x\left[\exp\left(-\int_0^{\tau_b^-} \omega(Z_s)ds\right); \tau^-_b<\infty  \right]=\frac{\sum^\infty_{n=0} W^{(\omega)}_{n}(x)}{\sum^\infty_{n=0} W^{(\omega)}_{n}(b)}.\]
	\end{theorem}
	
	\section{Main results and strategy of proofs}\label{results}
	From now on, unless explicitly stated, we consider the subcritical and critical nonlinear CSBPs, namely those satisfying $\gamma \geq 0$.  Recall $T_b:=\inf\{t\geq 0: X_t<b\}$ for any $b>0$ and the scale function $W$.
	\subsection{Main results}
	
	Our first result is the following necessary and sufficient condition for coming down from infinity.
	\begin{theorem}\label{main1}  Assume $\gamma\geq 0$. The boundary $\infty$ is an instantaneous entrance boundary for the process $(X_t,t\geq 0)$ if and only if \begin{equation}
		\label{cdicondition} \int^\infty\frac{1}{x\Psi(1/x)R(x)}\ddr x<\infty.
		\end{equation}
		Moreover, if \eqref{cdicondition} holds, then  for all  $b>0$, $\mathbb{E}_\infty(T_b)=m(b)<\infty$,
		where $m$ is defined on $(0,\infty)$ by
		\begin{equation}\label{m}
		m:b\mapsto \int_{b}^{\infty}\frac{W(x-b)}{R(x)}\ddr x.
		\end{equation}
	\end{theorem}
	\noindent The proof of Theorem \ref{main1} is deferred to Section \ref{firstentrance}.
	\begin{remark}
		One has $x\Psi\left(1/x\right)\underset{x\rightarrow \infty}{\longrightarrow} \Psi'(0+)=\gamma$. In the subcritical case, $\gamma>0$, and we see that \eqref{cdicondition} holds if and only if $\int^{\infty}\frac{\ddr x}{R(x)}<\infty.$  It is worth noticing that this integral condition  is the condition for $\infty$ to be an entrance for the deterministic flow $(x_t,t\geq 0)$ solving \eqref{ode}. In the critical case, $\gamma=0$, and the condition $\int^{\infty}\frac{\ddr x}{R(x)}<\infty$ is not sufficient for \eqref{cdicondition} to hold, but only necessary.
	\end{remark}

	We focus now on nonlinear CSBPs  satisfying \eqref{cdicondition}. In the subcritical case, for which $\gamma>0$, one defines the function $\varphi$ on $(0, \infty)$ as
	\begin{equation}\label{phi}
	\varphi:b\mapsto  \frac{1}{\gamma }\int_{b}^{\infty}\frac{\ddr x}{R(x)}.
	\end{equation}
	The following conditions will play a role in the sequel:
	$$\mathbb{H}_1: \underset{h\rightarrow 1+}{\limsup\ }\underset{x\rightarrow \infty}{\liminf} \frac{\varphi(hx)}{\varphi(x)}=1 \text{ and }\mathbb{H}_2:\text{ for any } h>1, \underset{x\rightarrow \infty}{\liminf} \frac{\varphi(x)}{\varphi(hx)}>1.$$
	Define for any $z>0$ and $\rho>0$, \begin{equation}\label{nobigvalley}V(z,\rho):=\sup_{x\geq z}\left(\frac{R(x)}{R(x+\rho z)}-1\right)_+
	\end{equation}
	where $(x)_+:=\max(x,0)$ is the positive part of $x$. The next condition can be interpreted as requiring the function $R$ to have no deep valleys near $\infty$: $$\mathbb{H}_3: \underset{\rho \rightarrow 0+}{\limsup } \underset{z\rightarrow \infty}{\lim} V(z,\rho)=0.$$
	Note that when $R$ is non-decreasing,  the function $V$ is identically zero and  $\mathbb{H}_3$ always holds.
	These three conditions frequently appear in the asymptotic analysis of stochastic processes, see e.g. Buldygin  et al. \cite{Buldyginbook} and the references therein.
	\begin{theorem}\label{regularCDI}  Assume $\gamma>0$ and \eqref{cdicondition}. If conditions $\mathbb{H}_1$, $\mathbb{H}_2$ and $\mathbb{H}_3$ are satisfied,  then
		in $\mathbb{P}_\infty$-probability, $\frac{X_t}{\varphi^{-1}(t)}\underset{t\rightarrow 0+}{\longrightarrow} 1.$
	\end{theorem}
	To exemplify Theorem \ref{regularCDI}, we will consider branching rate functions with \textit{regular variation}. Recall that  $R$ is  regularly varying at $\infty$ with index $\theta\in 	[0, \infty)$ if for any $\lambda>0$
	\begin{equation}\label{RV}\frac{R(\lambda x)}{R(x)}\underset{x\rightarrow \infty}{\longrightarrow}\lambda^{\theta}.\end{equation}
	The function $R$ is also called \textit{slowly varying} if  $\theta=0$.
	We refer the reader to Bingham et al. \cite{MR1015093} for a reference on those functions. Consider a function $R$, satisfying \eqref{RV} with  $\theta>1$. Karamata's theorem, see e.g. \cite[Proposition 1.5.10]{MR1015093}, ensures that
	\eqref{cdicondition} is satisfied and that the function $\varphi$ is regularly varying at $\infty$ with index $1-\theta<0$. This implies that  $\mathbb{H}_1$ and $\mathbb{H}_2$ are fulfilled.  Moreover, \cite[Theorem 1.5.12]{MR1015093} entails that $\varphi^{-1}$ is regularly varying at $0$ with index $-1/(\theta -1)$. If furthermore, $R$ is assumed to be non-decreasing, then $\mathbb{H}_3$  is also satisfied and Theorem \ref{regularCDI} applies. \\
	
	In the setting of a regularly varying branching rate $R$, some refinements of the condition $\mathbb{H}_3$ will enable us to get an almost-sure speed for two important classes of subcritical branching mechanisms. Recall $\pi$ the L\'evy measure of the underlying L\'evy process.
	%
	\begin{theorem}\label{regularvarR} Assume $\gamma>0$ and $R$ regularly varying at $\infty$ with index $\theta>1$.
		In both of the following cases (a) and (b)
		\begin{enumerate}
			\item[(a)] $\Psi(\lambda)-\gamma \lambda\sim c\lambda^{1+\delta} \,\,\, \text{ as}\,\,\, \lambda\rightarrow 0+$ for some $\delta\in (0,1)$, $c>0$ and $V(z,\frac{1}{\sqrt{z}})=O(z^{-\delta/2})\quad\mbox{as}~z\to\infty$,
			\item[(b)] there exists $\nu\in( 0,\infty)$ such that $\Psi(-\nu)=0$, $\int^1_0u\pi(\ddr u)<\infty$,  and $V(z,\frac{2 \ln \ln z}{\nu z})=O(\ln(z)^{-2}) \quad\mbox{as}~z\to\infty$,
		\end{enumerate}
		we have $\frac{X_t}{\varphi^{-1}(t)}\underset{t\rightarrow 0+}{\longrightarrow} 1$ $\mathbb{P}_\infty\text{-almost-surely}$
		and the speed function $\varphi^{-1}$ is regularly varying at $0$ with negative index $-{1}/{(\theta-1)}$.
	\end{theorem}
	\begin{remark} The first condition in \textrm{(b)} is known as Cram\'er condition, see e.g. Kyprianou \cite[Chapter 7]{MR3155252}. It requires that the L\'evy measure $\pi$ has an exponential moment.  The second condition on  $\pi$ means that there is no accumulation of small jumps. In particular, Cram\'er condition is satisfied in the pure diffusive case, for which for any $\lambda\geq 0$,  $\Psi(\lambda)= \gamma \lambda+\frac{\sigma^2}{2}\lambda^{2}$, and $(b)$ is fulfilled as soon as $V(z,\frac{2 \ln \ln z}{\nu z})=O(\ln(z)^{-2}) \quad\mbox{as}~z\to\infty$.  Also notice that if $R$ is increasing, then conditions of $V$ in \textrm{(a,b)} are fulfilled.
	\end{remark}
	In the following theorem, we study the critical processes started at $\infty$ when $\Psi$ is stable-like and $R$ is regularly varying at infinity.
	Recall the running infinum $(\underline{X}_t,t\geq 0)$ defined by $\underbar{X}_t=\inf_{0\le s\le t} X_s$. We find a renormalisation in law of  this running infimum. Recall $m$ defined in \eqref{m} and
	denote its right-inverse by $m^{-1}$.
	\begin{theorem}\label{stablecriticalrv}
		Assume that $\gamma=0$,  $R$ is regularly varying at $\infty$ with index $\theta>\alpha$, and $\Psi(\lambda)\sim c\lambda^{\alpha}$ as $\lambda$ goes to $0$ for $\alpha\in (1,2]$ and $c>0$. Then $\infty$ is an instantaneous entrance boundary and under $\mathbb{P}_\infty$
		\[\limsup_{t\rightarrow 0+}\frac{X_t}{{\underline X}_t}=\infty \text{ a.s and }
		\frac{\underbar{X}_t}{m^{-1}(t)}{\longrightarrow} S_{\alpha,\theta}^{\frac{1}{\theta-\alpha}} ~\text{ in law
			as } t \text{ goes to } 0,\]
		where $m^{-1}$ is regularly varying at $0$ with negative index $-\frac{1}{\theta-\alpha}$,
		and $S_{\alpha,\theta}$ is  a positive random variable with Laplace transform
		\[\mathbb{E}[e^{-sS_{\alpha,\theta}}]=\left[\sum^\infty_{n=0}\left( \frac{s\Gamma(\theta)}{\Gamma(\theta-\alpha)}\right)^n\prod_{i=1}^{n}\frac{\Gamma(i\theta-i\alpha)}{\Gamma(i\theta-(i-1)\alpha)} \right]^{-1}\!\!\!\!\! \text{ for any } s\geq 0,\]
		where $\Gamma$ is the Gamma function and the empty product, with $n=0$, is taken to be $1$.
	\end{theorem}
	
	Theorem \ref{stablecriticalrv} naturally leads to the question whether the critical processes that come down from infinity,  always oscillate widely or not. In the next proposition, we find a class of functions $R$ for which $\mathbb{H}_1$ is not satisfied and the running infimum of the critical process can be renormalized to converge in probability towards $1$.
	\begin{prop}\label{fastCDI}  Assume that $R(x)=g(x)e^{\theta x}$ for $x>0$ and some constant $\theta>0$, and that the function $g$ is continuous on $[0, \infty)$, strictly positive on $(0,\infty)$ and regularly varying at $\infty$. Then
		$\infty$ is an instantaneous entrance boundary and in $\mathbb{P}_\infty$-probability,
		$$\frac{\underline X_t}{m^{-1}(t)}\underset{t\rightarrow 0}{\longrightarrow} 1 \text{ in the critical case and }
		\frac{X_t}{\varphi^{-1}(t)}\underset{t\rightarrow 0}{\longrightarrow} 1 \text{ in the subcritical case}.$$
		Moreover, the speed functions $t\mapsto \varphi^{-1}(t)$ and $t\mapsto m^{-1}(t)$ are slowly varying at $0$. 
	\end{prop}
	Theorems \ref{regularCDI}, \ref{regularvarR}, \ref{stablecriticalrv} and Proposition \ref{fastCDI} are proved  in Section 6. \\
	
	We show now how our results apply for explicit branching rates and/or branching mechanisms.
	\begin{example}\label{examples}\
		\begin{enumerate}
			\item [1)] If for any $x\geq 0$, $R(x)=x^{\theta}$ for some $\theta>0$, then a simple change of variable in \eqref{cdicondition} ensures that $\infty$ is an entrance boundary if and only if $\int_{0}\frac{x^{\theta-1}}{\Psi(x)}\ddr x<\infty.$ This condition was established in Li \cite[Theorem 1.10]{2016arXiv160909593L}. Recall that $\int_{0}\frac{1}{\Psi(x)}\ddr x=\infty$ for any $\Psi$ defined as in \eqref{LK}. When $\theta=1$, we recover that (sub)critical CSBPs do not come down from infinity. However when $\theta>1$ and $\gamma>0$, by Theorem \ref{regularCDI}, the process comes down from infinity at speed
			$\varphi^{-1}(t)=[\gamma(\theta-1)t]^{\frac{1}{1-\theta}}.$
			\item [2)] If for any $x\geq 0$, $\Psi(x)=cx^{\alpha}$ for some $c>0$ and $\alpha \in (1,2]$, then $\infty$ is an entrance boundary if and only if $\int^{\infty}\frac{x^{\alpha-1}}{R(x)}\ddr x<\infty,$  see Example 4.17 of \cite{MR3014147} for the scale function for $\Psi$.
			In this case, the nonlinear CSBP $X$ is also a stable jump diffusion (with asymmetric jumps), see D\"oring and Kyprianou \cite{2018arXiv180201672D}.
			We recover the integral test established in  \cite[Theorem 2.2]{2018arXiv180201672D}.
			When for any $x\geq 0$,  $R(x)=x^\theta$ with $\theta>\alpha$, the process comes down from infinity, and by Theorem \ref{stablecriticalrv}, its running infimum  comes down at speed $m^{-1}(t)=[\Gamma(\theta)/\Gamma(\theta-\alpha)t]^{\frac{1}{\alpha-\theta}}.$ See the forthcoming Remark \ref{exactstable} for the explicit form of $m$.
			\item [3)]
			
			If for any $x\geq 0$, $R(x)=e^{\theta x}$ for some $\theta>0$, then for any subcritical or critical branching mechanism $\Psi$, by Lemma \ref{l1}, $\int^{\infty}\frac{\ddr x}{R(x)\Psi(1/x)}<\infty$ and by Theorem \ref{main1}, the process comes down from infinity.  By Proposition \ref{fastCDI}, if $\gamma>0$,  then the speed is $\varphi^{-1}(t)=-\log(\gamma \theta t)/\theta\underset{t\rightarrow 0+}{\sim} -\log (t)/\theta.$  If  $\gamma=0$,  then the running infimum comes down at speed $m^{-1}(t)=-\log(\Psi(\theta) t)/\theta \underset{t\rightarrow 0+}{\sim} -\log (t)/\theta$. See the forthcoming Corollary \ref{p9} for an explicit form of $m$.
		\end{enumerate}
	\end{example}
	An example of nonlinear CSBP with a non-increasing branching rate R is studied at the end of Section 5, see  Example \ref{examplenonincreas}.
	\subsection{Strategy of proofs}
	
	In order to establish our main results, we will follow  the approach developed by Bansaye  et al. in  \cite{2017arXiv171108603B}, \cite{MR3595771} for Kolmogorov diffusions and  birth-death processes. It consists in studying the long-term behavior of $T_b$ under $\mathbb{P}_\infty$ when $b$ goes to $\infty$. The main difference with these works lies in the fact that the process $X$ may have arbitrarily large jumps, which makes its position more involved to follow.
	
	Let $b>0$. Recall  $\tau_b^-:=\inf\{t\geq 0: Z_t<b\}$.  A simple time-change argument gives  \beqlb\label{0.6}
	T_b:=\int^{\tau_b^-}_0 \frac{1}{R(Z_s)}\ddr s.
	\eeqlb
	We see therefore that $T_b$ is a weighted occupation time for the spectrally positive L\'evy process $Z$ with the function $\omega=1/R$. We will apply the known  results on those functionals of L\'evy processes, as well as Theorem \ref{occupation} for studying the moments and the Laplace  transform under $\mathbb{P}_\infty$.
	
	Recall the function $\varphi$ defined in \eqref{phi}. The map $\varphi$ is strictly decreasing and its inverse function $\varphi^{-1}$ is  the solution $(x_t,t\geq 0)$ to the Cauchy problem \eqref{ode} started from $x_0=\infty$. For any $b>0$, $\varphi(b)$ corresponds to the first passage time of $(x_t,t\geq 0)$ below the level $b$ and $\varphi(b)-\varphi(bh)$  is the time needed for $(x_t,t\geq 0)$ to go from the level $bh$ to $b$ for any $b>0$ and $h>1$. Observe that $\mathbb{H}_1$ is equivalent to
	$\underset{h\rightarrow 1+}{\liminf}\ \underset{b\rightarrow \infty}{\limsup}\ \frac{\varphi(b)-\varphi(bh)}{\varphi(b)}=0$.
	Intuitively, the time needed for  $(x_t,t\geq 0)$ to go from $bh$ to $b$ with $h>1$, is  negligible in comparison to the time needed to reach $b$ when started at $\infty$. The condition $\mathbb{H}_1$ requires therefore that the coming down from infinity does not occur too fast.
	
	We shall see that when $\mathbb{H}_1$ is satisfied, the mean time for the process $X$, started from $\infty$, to reach $b$ is equivalent to $\varphi(b)$. Under the conditions $\mathbb{H}_1$ and $\mathbb{H}_3$, the following weak law of large numbers will occur
	\begin{equation}\label{wlln}
	\frac{T_b}{\mathbb{E}_\infty[T_b]}\underset{b\rightarrow \infty}{\longrightarrow} 1 \text{ in } \mathbb{P}_\infty\text{-probability}.
	\end{equation}
	
	The additional condition $\mathbb{H}_2$ on $\varphi$ will allow us to transfer our results on $T_b$ as $b$ goes to $\infty$  to results on the small-time asymptotics under $\mathbb{P}_\infty$, of the running infimum process $(\underline{X}_t,t\geq 0)$, defined by $\underbar{X}_t=\inf_{0\le s\le t} X_s$. In the subcritical case, excursions of $(X_t,t\geq 0)$ above its running infimum are negligible and asymptotics for $(\underline{X}_t,t\geq 0)$ will provide asymptotics for $(X_t,t\geq 0)$.
	
	Getting the  almost-sure convergence in \eqref{wlln} is more involved. The method   requires rather explicit fluctuation identities, which are available for important classes of branching mechanisms, namely those with stable-like behaviors and with exponential moments.
	
	In the critical case  or when $\mathbb{H}_1$ is not satisfied, the convergence in probability \eqref{wlln} will typically not occur. We shall see that for the critical branching mechanisms considered in Theorem \ref{stablecriticalrv},
	$\left({T_b}/{\mathbb{E}_\infty[T_b]}, b\geq 0\right) $ converges in law as $b$ goes to $\infty$.

	\section{First entrance times}\label{firstentrance}
	We prove in this section Theorem \ref{main1} and provide a formula for the variance of the first entrance time $T_b$ under $\mathbb{P}_\infty$. Recall the expression \eqref{0.6} for $T_b$ under $\mathbb{P}_x$ for any $x\in (0,\infty)$.
	\blemma\label{l4}
	For any $x>b>0$ we have
	$\mbb{E}_x [T_b]
	=\int^\infty_{b}\frac{\ddr y}{R(y)} [W(y-b)-W(y-x)].$
	\elemma
	\begin{proof}
		By the time change  we have
		\begin{equation*}
		\begin{split}
		\mbb{E}_x [T_b]&=\mbb{E}_x\Big[\int_0^{\tau_b^-}\frac{\ddr s}{R(Z_s)}\Big]=\int_0^{\infty}\frac{1}{R(y)}\int_0^\infty\mbb{P}_x(Z_s\in \ddr y, s<\tau^-_b)\ddr s\\
		&=\int_0^{\infty}\frac{W(y-b)-W(y-x)}{R(y)}\ddr y,
		\end{split}
		\end{equation*}
		where for the last equality, we apply Kuznetsov et al. \cite[Theorem 2.7 (ii)]{MR3014147} for an expression of the potential measure.
	\end{proof}
	
	\begin{proof}[\noindent \textbf{Proof of Theorem \ref{main1}}]
		We have seen in Lemma \ref{supercritical} that in the supercritical case (for which $\gamma<0$), $\infty$ is not an entrance boundary.  We treat now the case $\gamma\geq 0$. Recall that there is no explosion when $\gamma\geq 0$. Since the scale function is continuous and strictly increasing on $[0,\infty)$, using Lemma \ref{l4} and the monotone convergence theorem, we see that
		\begin{equation}\label{moment1}
		\underset{x\rightarrow \infty}{\lim} \mbb{E}_x [T_b]=\int^\infty_{b}\frac{W(y-b)}{R(y)}\ddr y=m(b).
		\end{equation}
		By Lemma \ref{equivcdi} (b),  $\infty$ is an entrance boundary, namely (\ref{0.5'}) holds, if and only if $m(b)<\infty$ for a large enough $b>0$. By Lemma \ref{l1}, $\frac{W(y)}{W(y-b)}\underset{y\rightarrow \infty}{\longrightarrow} 1$ for any $b>0$ and using the bounds (\ref{boundW}), one sees that the last integral is finite if and only if (\ref{cdicondition}) holds.   Moreover, by applying Proposition \ref{momentunderPinfty} with the function $h(x)=x$, we obtain that for any $b>0$ such that $m(b)<\infty$, one has $m(b)=\mathbb{E}_\infty[T_b]$.
		
		It remains to establish that if there is $c$ such that $m(c)<\infty$ then $m(b)<\infty$ for all $b>0$. Suppose that for a constant $c>0$, $\mbb{E}_\infty (T_c)<\infty$. Then by Lemma \ref{l1}, $\frac{W(y-c)}{W(y-b)}\underset{y\rightarrow \infty}{\longrightarrow} 1$ for $b>0$, and thus, from \eqref{moment1} we see that for any constant $b>0$,  $\mbb{E}_\infty [T_b]<\infty$.
	\end{proof}
	\begin{lemma} Under the assumption  \eqref{cdicondition}, the function $m$ is positive, continuous and strictly decreasing on $(0,\infty)$.
	\end{lemma}
	\begin{proof}
		By the strong Markov property and the absence of negative jumps, for any $b<a$, $m(b)-m(a)=\mathbb{E}_{\infty}[T_b-T_{a}]=\mathbb{E}_{a}[T_b]>0$
		and $m$ is strictly decreasing. Let $a>0$ such that $x\mapsto \frac{W(x)}{R(x)}$ is integrable on $(a,\infty)$. For any $x>a$, $\frac{W(x-b)}{R(x)}\leq \frac{W(x)}{R(x)}$. Since the map $b\mapsto \frac{W(x-b)}{R(x)}$ is continuous,  by continuity under the integral sign, the map $m$ is continuous on $(a,\infty)$.
	\end{proof}
	%
	%
	We now study the Laplace transform of the first entrance time $T_b$. Define recursively the positive functions $(W_n,n\geq 0)$ by
	\begin{equation}\label{Wn}
	W_0(x)=1,\quad W_{n+1}(x)=\int_{x}^{\infty}\frac{W(z-x)}{R(z)}W_n(z)\ddr z,\qquad n\ge0, x\geq 0.
	\end{equation}
	Note that  $W_1(b)=m(b)$ for any $b>0$.
	\begin{theorem}\label{c7}  If there is $b\geq 0$ such that $m(b)<\infty$, then for any $0<\lambda<1/m(b)$, $\sum^\infty_{n=0} \lambda^n W_{n}(b)<\infty$ and for any $x>b$, we have
		\beqlb\label{c10}\mbb{E}_x\left[e^{-\lambda T_b}\right]=\frac{\sum^\infty_{n=0}\lambda^n W_{n}(x)}{\sum^\infty_{n=0} \lambda^n W_{n}(b)}.\eeqlb
		In addition,
		\begin{equation}\label{Wninf}
		\mbb{E}_\infty \left[e^{-\lambda T_b}\right]=\frac{1}{\sum^\infty_{n=0} \lambda^n W_{n}(b)}.
		\end{equation}
	\end{theorem}
	%
	\begin{proof}
		We show that $\sum_{n\geq 0}\lambda^n W_n(b)<\infty $ for $0<\lambda<\frac{1}{W_{1}(b)}$. Recall from Lemma \ref{l1} that $x\mapsto W(x)$ is nondecreasing. By induction we can show that for any $n$, the function $x\in [0,\infty)\mapsto W_n(x)$ is decreasing. Then
		\begin{align*}
		W_{n+1}(x)&=\int_{x}^{\infty}\frac{W(z-x)}{R(z)}W_n(z)\ddr z
		\leq W_n(x)W_1(x)
		\end{align*}
		so that $W_n(x)\leq W_1(x)^{n}$ for any $x\geq 0$. Since
		$W_1(x)\underset{x\rightarrow \infty}{\longrightarrow} 0$,  for any $\lambda>0$, there exists $b>0$ such that $\lambda<\frac{1}{W_{1}(b)}$. For any $x\geq b$, \[\sum_{n\geq 0}\lambda^{n}W_n(x)\leq \sum_{n\geq 0}\lambda^{n}W_n(b)\leq \sum_{n\geq 0}(\lambda W_1(b))^{n}<\infty.\]
		By applying Theorem~\ref{occupation} with $\omega(x)=\frac{1}{R(x)}$, and noticing that $\lim_{x\to\infty}W_n(x)=0$, $n\ge 1$, we obtain the series representation (\ref{c10}).
	\end{proof}
	According to Theorem \ref{occupation}, the formula for the Laplace transform \eqref{c10} holds true more generally at any $\lambda>0$ such that the series $\sum_{n\geq 1}\lambda^{n}W_n(b)$ converges.
	
	We compute explicitly the functions $(W_n,n\geq 1)$ when $R(x)=e^{\theta x}$. They were first obtained by Patie in \cite{MR2548498} in the theory of PSSMPs.
	\begin{corol}\label{p9}
		Suppose that $R(x)=e^{\theta x}$ for a given $\theta>0$ and $\gamma\geq 0$. Then for any $b\geq 0$, $m(b)=\frac{e^{-\theta b}}{\Psi(\theta)}<\infty$ and for any $x\geq 0$,
		\begin{equation}\label{Patie}
		W_n(x):=\frac{e^{-n\theta x}}{\prod^n_{j=1}\Psi(j \theta)},\qquad \text{for} \,\, n\ge1.
		\end{equation}
		Further, \eqref{c10} and \eqref{Wninf} hold true  for any $\lambda>0$.
	\end{corol}
	\begin{proof}
		We are going to prove it by induction. By Lemma \ref{l1}, for any $x\geq 0$, $m(x)=W_1(x)=\int_{0}^{\infty}\frac{W(y)}{e^{\theta(y+x)}}\ddr y=\frac{e^{-\theta x}}{\Psi(\theta)}<\infty$.
		Suppose that \eqref{Patie} holds for $n=m$. Then for $n=m+1$ we have
		\beqnn
		W_{m+1}(x)
		\ar=\ar
		\int^\infty_x\frac{W(z-x)}{e^{\theta z}}W_m(z)\ddr z=
		\int^\infty_0\frac{W(z)}{e^{\theta(z+x)}}W_m(z+x)\ddr z\cr
		\ar=\ar
		\frac{1}{\prod^m_{j=1}\Psi(j \theta)}\int^\infty_0W(z)e^{-(m+1)\theta (z+x)}\ddr z
		=
		\frac{e^{-(m+1)\theta x}}{\prod^{m+1}_{j=1}\Psi(j \theta)},\cr
		\eeqnn
		where we used (\ref{defscale}) in the third equality.
		The formula is obtained for any $n$ by induction.
		
		One readily checks that for any $x\geq 0$, $\frac{W_{n+1}(x)}{W_{n}(x)}\underset{n\rightarrow \infty}{\longrightarrow} 0$, which ensures that the entire series with coefficients \eqref{Patie} has an infinite radius of convergence.
	\end{proof}

	We see in the next proposition that in the subcritical case the series \eqref{c10} always converges.
	\begin{prop}\label{all-lambda}
		Suppose that $\gamma>0$ and $\int^{\infty}\frac{\ddr x}{R(x)}<\infty$. Then, for all $b>0$, for any $\lambda>0$, $\sum_{n=0}^\infty \lambda^n W_n(b)<\infty$.
	\end{prop}
	\begin{proof}
		Recall the function $\varphi$, defined in \eqref{phi}, as $\varphi(z)=\frac{1}{\gamma}\int_{z}^{\infty}\frac{\ddr u}{R(u)}$.  We show by induction and (\ref{Wn}) that for all $n\geq 1$ and all $b>0$, $W_n(b)\leq \frac{\varphi(b)^n}{n!}$. Since for any $x\geq 0$, $W(x)\leq W(\infty)=\frac{1}{\gamma}$, one clearly has $W_1(b)=\int_{b}^{\infty}\frac{W(x-b)}{R(x)}\ddr x\leq \varphi(b)$.
		Assume that for any $b>0$, $W_n(b)\leq \frac{\varphi(b)^n}{n!}$. The recursion formula
		\eqref{Wn} entails
		\[W_{n+1}(b)=\int_{b}^{\infty}\frac{W(z-b)}{R(z)}W_n(z)\ddr z\leq \frac{1}{n!}\int_{b}^{\infty}\frac{1}{\gamma R(z)}\varphi(z)^n \ddr z. \]
		Since $\varphi'(z)=-\frac{1}{\gamma R(z)}$, then
		$\int_{b}^{\infty}\frac{\varphi(z)^n}{\gamma R(z)}\ddr z=\frac{\varphi(b)^{n+1}}{n+1}$
		and $W_{n+1}(b)\leq \frac{\varphi(b)^{n+1}}{(n+1)!}$. The result follows by induction.
	\end{proof}
	\blemma\label{l6} Suppose that \eqref{cdicondition} holds. Then for any $x\geq b>0$,
	$\mbb{E}_x[T_b^2]=2[W_2(x)-W_2(b)+W_1(b)(W_1(b)-W_1(x))]$,
	which can also be  written as
	\beqnn
	\mbb{E}_x[T_b^2]=2\int^\infty_b\frac{\ddr u}{R(u)}
	\int^\infty_b\frac{\ddr z}{R(z)}[W(z-b)-W(z-x)][W(u-b)-W(u-z)].
	\eeqnn
	In particular, we have
	\beqnn
	\mbb{E}_\infty[T_b^2]=2W_1(b)^2-2W_2(b)
	\ar=\ar
	2\left[\int^\infty_{b}\frac{1}{R(y)}W(y-b)\ddr y\right]^2\cr
	\ar\ar
	-2\int^\infty_{b}\frac{1}{R(x)}\ddr x\int^\infty_{b}\frac{1}{R(y)}W(y-b)W(x-y)\ddr y.
	\eeqnn
	\elemma
	\begin{proof}
		We have seen in the proof of Theorem \ref{c7}, that for any {$x\ge 0$}, $W_n(x)\leq W_1(x)^{n}$ for all $n\geq 0$. Hence, for any $0\leq \lambda<\frac{1}{m(b)}$, $\sum_{n=2}^{\infty}n(n-1)\lambda^{n-2}W_n(x)<\infty$. For any $x>b$ and  $0\le\lambda<1/m(b)$, set $f_x(\lambda):=\sum_{n=0}^{\infty}\lambda^{n}W_n(x)$ and note that $\mathbb{E}_x[e^{-\lambda T_b}]=\frac{f_x(\lambda)}{f_b(\lambda)}$.  The function $f_x$ is twice differentiable and satisfies $f_x(0)=1$, $f'_x(0)=W_1(x)$ and $f''_x(0)=2W_2(x)$. Simple computations provide  that for $0\le\lambda<1/m(b),$
		\begin{equation*}
		\begin{split}
		\frac{\ddr^2 }{\ddr \lambda^2}\mathbb{E}_x[e^{-\lambda T_b}]=
		&(f''_x(\lambda)f_b(\lambda)-f_x(\lambda)f''_b(\lambda))f_b(\lambda)^{-2}\\
		&-2(f'_x(\lambda)f_b(\lambda)-f_x(\lambda)f_b'(\lambda))f_b(\lambda)^{-3}f'_b(\lambda)
		\end{split}
		\end{equation*}
		and we get with $\lambda=0$, $\mathbb{E}_x[T_b^2]=2[W_2(x)-W_2(b)+W_1(b)(W_1(b)-W_1(x))]$.
		Let $I_1:=W_2(x)-W_2(b)$ and let $I_2:=W_1(b)(W_1(b)-W_1(x))$.
		Recall $W_1$ and $W_2$ defined in (\ref{Wn}) and $W_1(b)=m(b).$
		Noticing the fact that $W(x)=0$ for $x<0$, we see that
		\beqnn
		I_1\ar=\ar\int^\infty_b\frac{W(z-x)}{R(z)}\int^\infty_z\frac{W(u-z)}{R(u)}\ddr u\ddr z
		-\int^\infty_b\frac{W(z-b)}{R(z)}\int^\infty_z\frac{W(u-z)}{R(u)}\ddr u\ddr z\cr
		\ar=\ar
		\int^\infty_b\frac{\ddr u}{R(u)}\int^u_b\frac{W(z-x)W(u-z)}{R(z)}\ddr z
		-\int^\infty_b\frac{\ddr u}{R(u)}\int^u_b\frac{W(z-b)W(u-z)}{R(z)}\ddr z\cr
		\ar=\ar
		-\int^\infty_b\frac{\ddr u}{R(u)}\int^\infty_b \frac{\ddr z}{R(z)}[W(z-b)-W(z-x)]W(u-z)
		\eeqnn
		and
		$I_2=\int^\infty_b\frac{\ddr u}{R(u)}W(u-b)\int^\infty_b\frac{\ddr z}{R(z)}[W(z-b)-W(z-x)]$.
		Therefore,
		\beqnn
		\mbb{E}_x [T_b^2]=2I_1+2I_2=2\int^\infty_b\frac{\ddr u}{R(u)}
		\int^\infty_b\frac{\ddr z}{R(z)}[W(z-b)-W(z-x)][W(u-b)-W(u-z)].
		\eeqnn
		Finally, we let $x\to\infty$. Then by the fact that $\lim_{x\to0+} W(x)=0$ decreasingly and Proposition \ref{momentunderPinfty}  (b), we obtain the desired result.
	\end{proof}
	\bcorollary\label{c4} Suppose that \eqref{cdicondition} holds.
	For any $b>0$,
	$\text{Var}_\infty(T_b)=W_1(b)^2-2W_2(b)$
	which can also be  written as
	\beqlb\label{1.4}
	\text{Var}_\infty(T_b)=2\int^\infty_{b}\frac{1}{R(x)} W(x-b)\ddr x\int^\infty_{x}\frac{1}{R(y)}[W(y-b)-W(y-x)]\ddr y.
	\eeqlb
	\ecorollary
	\begin{proof}
		By Lemmas~\ref{l4} and~\ref{l6},
		\beqnn
		\text{Var}_\infty(T_b)\ar=\ar{\mbb{E}_\infty[T_b^2]-(\mbb{E}_\infty [T_b])^2} \cr
		\ar=\ar
		\int^\infty_{b}\frac{1}{R(x)}\ddr x\int^\infty_{b}\frac{1}{R(y)}W(y-b)W(x-b)\ddr y\cr
		\ar\ar
		-2\int^\infty_{b}\frac{1}{R(x)}\ddr x\int^\infty_{b}\frac{1}{R(y)}W(y-b)W(x-y)\ddr y\cr
		\ar=\ar
		2\int^\infty_{b}\frac{1}{R(x)}\ddr x\int^\infty_{x}\frac{1}{R(y)}W(y-b)W(x-b)\ddr y\cr
		\ar\ar
		-2\int^\infty_{b}\frac{1}{R(x)}\ddr x\int^x_{b}\frac{1}{R(y)}W(y-b)W(x-y)\ddr y.
		\eeqnn
		Changing the order of integrals we have
		\beqnn
		\text{Var}_\infty(T_b)
		\ar=\ar
		2\int^\infty_{b}\frac{1}{R(x)}\ddr x\int^\infty_{x}\frac{1}{R(y)}W(x-b)W(y-b)\ddr y\cr
		\ar\ar
		-2\int^\infty_{b}\frac{1}{R(x)}\ddr x\int^\infty_{x}\frac{1}{R(y)}W(x-b)W(y-x)\ddr y\cr
		\ar=\ar
		2\int^\infty_{b}\frac{1}{R(x)} W(x-b)\ddr x\int^\infty_{x}\frac{1}{R(y)}[W(y-b)-W(y-x)]\ddr y.
		\eeqnn
	\end{proof}

	\section{Asymptotic behaviors of hitting times}\label{asymptoticsTb}
	In this section, we study the convergence of  $(T_b/\mathbb{E}_\infty[T_b],b\geq 0)$ as $b\rightarrow\infty$. 
	By applying Theorem \ref{c7}, we first find conditions on $R$ for a convergence in law to hold.
	\subsection{Convergence in law}

	\begin{corol}\label{exponen_conve}
		Suppose that  $R(x)= e^{\theta_2 x}g(x)$ for {$x>0$ and} some constant $\theta_2>0$ and function $g$ is  regularly varying at $\infty$. Then \eqref{cdicondition} is satisfied and $m(b)\sim \frac{1}{\Psi(\theta_2) R(b)}$ as $b$ goes to $\infty$. Moreover, for any $\lambda\geq 0$,
		$$
		\lim_{b\to\infty}\mbb{E}_\infty\left[e^{-\lambda\frac{T_b}{m(b)}}\right]=\Big({1+\sum_{n=1}^\infty  \frac{\Psi(\theta_2)^{n}\lambda^n}{\prod_{j=1}^{n}\Psi(j\theta_2)}}\Big)^{-1}.
		$$
	\end{corol}
	\begin{proof}
		Assume that $g$ is regularly varying with index $\theta_1\in \mathbb{R}$. Set $\ell(x):=x^{-\theta_1}g(x)$ for any $x>0$. The function $\ell$ is slowly varying at $\infty$. Recall (\ref{Wn}) and $W_1(x)=m(x)$.
		For any $x>0$
		\begin{equation}\label{up1}
		W_1(x)=\int_0^\infty \frac{W(z)}{\ell(x+z)(x+z)^{\theta_1}e^{\theta_2 (x+z)}}\ddr z
		\leq \frac{e^{-\theta_2 x}}{\ell(x)x^{\theta_1}} \int_{0}^{\infty} \frac{W(z)}{e^{\theta_2 z}}\frac{\ell(x)}{\ell(x+z)}\ddr z.
		\end{equation}
		The  representation  theorem for {slowly varying functions}, see e.g. \cite[Theorem 1.3.1]{MR1015093}, entails that for any fixed {$z\ge0$}, $\frac{\ell(x)}{\ell(x+z)}\underset{x\rightarrow \infty}{\longrightarrow} 1$. Moreover, by Potter's theorem, see e.g. \cite[Theorem 1.5.6]{MR1015093}, for  any chosen constant $C>1$, and large enough $x$,
		\begin{equation}\label{potter}
		\frac{\ell(x)}{\ell(x+z)}\leq C\left(1+\frac{z}{x}\right)\leq C(1+z).
		\end{equation}
		Recall \eqref{defscale}, $\int_{0}^{\infty} \frac{W(z)}{e^{\theta_2 z}}\ddr z=\frac{1}{\Psi(\theta_2)}$. Fix $\epsilon\in(0,1)$. Since $\int_{0}^{\infty}(1+z)W(z)e^{-\theta_2 z}\ddr z<\infty$,  by Lebesgue's theorem, for $x$ large enough
		$\int_{0}^{\infty}\frac{W(z)}{e^{\theta_2 z}}\frac{\ell(x)}{\ell(x+z)}\ddr z\leq \frac{1+\epsilon}{\Psi(\theta_2)}$
		and (\ref{up1}) entails $W_1(x)\leq \frac{1+\epsilon}{R(x)\Psi(\theta_2)}$. In order to find a lower bound, first note that for $0<\beta<1$
		$W_1(x)\geq \int_0^{x^{\beta}} \frac{W(z)}{\ell(x+z)(x+z)^{\theta_1}e^{\theta_2 (x+z)}}\ddr z$.
		Applying  \eqref{defscale} again, we have for large enough $x$
		\begin{equation}\label{firstbound}
		\int_{0}^{x^{\beta}} \frac{W(z)}{e^{\theta_2 z}}\ddr z\geq \frac{1-\epsilon}{\Psi(\theta_2)}.
		\end{equation}
		Thus,
		\begin{equation}
		\begin{split}
		\int_0^{x^{\beta}} \frac{W(z)}{\ell(x+z)(x+z)^{\theta_1}e^{\theta_2 (x+z)}}\ddr z&\geq \frac{e^{-\theta_2 x}}{\sup_{y\in [x,x+x^{\beta}]}\ell(y)(x+x^{\beta})^{\theta_1}} \int_{0}^{x^{\beta}} \frac{W(z)}{e^{\theta_2 z}}\ddr z\\
		&\geq \frac{1}{R(x)}\frac{\ell(x)}{\sup_{y\in [x,x+x^{\beta}]}\ell(y)}\left(\frac{x}{x+x^{\beta}}\right)^{\theta_1}\frac{1-\epsilon}{\Psi(\theta_2)}.
		\end{split}
		\end{equation}
		
		Since $\beta<1$, $\left(\frac{x}{x+x^{\beta}}\right)^{\theta_1}\underset{x\rightarrow \infty}{\longrightarrow} 1$ and for $x$ large enough $$\frac{\ell(x)}{\inf_{\lambda\in [1,2]}\ell(\lambda x)}\ge \frac{\ell(x)}{\sup_{y\in [x,x+x^{\beta}]}\ell(y)}\ge\frac{\ell(x)}{\sup_{\lambda\in [1,2]}\ell(\lambda x)}.$$
		Applying the uniform convergence theorem, \cite[Theorem 1.2.1]{MR1015093}, to the slowly varying functions $\ell$ and $1/\ell$, we see that both of the bounds above converge towards $1$ as $x$ goes to $\infty$. Therefore, for large enough $x$, \begin{equation}\label{UCT}\frac{\ell(x)}{\sup_{y\in [x,x+x^{\beta}]}\ell(y)}\left(\frac{x}{x+x^{\beta}}\right)^{\theta_1}\geq 1-\epsilon
		\end{equation}
		and thus, for large enough $x$
		\begin{equation}\label{W1encadrement}
		\frac{(1-\epsilon)^2}{R(x)\Psi(\theta_2)}\leq W_1(x)\leq \frac{1+\epsilon}{R(x)\Psi(\theta_2)}.
		\end{equation}
		Recall that $W_1(x)=m(x)$. The inequalities above yield that $m(x)\sim \frac{1}{\Psi(\theta_2)R(x)}$ as $x$ goes to infinity. We proceed to show by induction that for any $n\geq 1$, there is $m_n\in \mathbb{N}$ such that  for all large enough $x$,
		\begin{equation}\label{bounds}
		\frac{(1-\epsilon)^{m_n}}{R(x)^n\prod_{j=1}^{n}\Psi(j\theta_2)}\leq W_n(x)\leq \frac{(1+\epsilon)^n}{R(x)^n\prod_{j=1}^{n}\Psi(j\theta_2)},
		\end{equation}
		which immediately yields
		$W_n(x)\sim \frac{1}{R(x)^n\prod_{j=1}^{n}\Psi(j\theta_2)}
		$ as $x$ goes to infinity. By letting $m_1=2$ one can see that  (\ref{bounds}) holds for $n=1$. Assume that (\ref{bounds}) holds for some $n\geq 1$ and $m_n\geq 1$.
		To show (\ref{bounds}) for $n+1$  we start with the lower bound.
		\begin{align*}
		W_{n+1}(x)&=\int_{0}^{\infty}\frac{W(z)W_n(x+z)}{R(x+z)}\ddr z\geq \frac{(1-\epsilon)^{m_n}}{\prod_{j=1}^{n}\Psi(j\theta_2)}\int_{0}^{\infty}\frac{W(z)}{R(x+z)^{n+1}}\ddr z\\
		&\geq \frac{(1-\epsilon)^{m_n}}{\prod_{j=1}^{n}\Psi(j\theta_2)}\int_{0}^{x^{\beta}}\frac{W(z)}{R(x+z)^{n+1}}\ddr z\\
		&=\frac{(1-\epsilon)^{m_n}}{R(x)^{n+1}\prod_{j=1}^{n}\Psi(j\theta_2)}\int_{0}^{x^{\beta}}\frac{W(z)}{e^{(n+1)\theta_2 z}}\left(\frac{x}{x+z}\right)^{(n+1)\theta_1}\left(\frac{\ell(x)}{\ell(x+z)}\right)^{n+1}\ddr x.
		\end{align*}
		Applying (\ref{firstbound}) for $(n+1)\theta_2$ and (\ref{UCT}), we obtain that
		\begin{align*}
		W_{n+1}(x)&\geq \frac{(1-\epsilon)^{m_n}}{R(x)^{n+1}\prod_{j=1}^{n}\Psi(j\theta_2)}\frac{1-\epsilon}{\Psi((n+1)\theta_2)}(1-\epsilon)^{n+1}
		=\frac{(1-\epsilon)^{m_{n+1}}}{R(x)^{n+1}\prod_{j=1}^{n+1}\Psi(j\theta_2)}
		\end{align*}
		with $m_{n+1}:=m_n+n+2$. We now look for the upper bound. One has
		\begin{align*}
		W_{n+1}(x)&\leq \frac{(1+\epsilon)^{n}}{\prod_{j=1}^{n}\Psi(j\theta_2)}\int_{0}^{\infty}\frac{W(z)}{R(x+z)^{n+1}}\ddr z\\
		&\leq \frac{(1+\epsilon)^{n}}{R(x)^{n+1}\prod_{j=1}^{n}\Psi(j\theta_2)}\int_{0}^{\infty}\frac{W(z)}{e^{(n+1)\theta_2 z}}\left(\frac{\ell(x)}{\ell(x+z)}\right)^{n+1}\ddr z.
		\end{align*}
		By (\ref{potter}), $\left(\frac{\ell(x)}{\ell(x+z)}\right)^{n+1}\leq C^{n+1}(1+z)^{n+1}$ and since $\int_{0}^{\infty}(1+z)^{n+1}\frac{W(z)}{e^{(n+1)\theta_2 z}}\ddr z<\infty$, Lebesgue's theorem entails
		$$\int_{0}^{\infty}\frac{W(z)}{e^{(n+1)\theta_2 z}}\left(\frac{\ell(x)}{\ell(x+z)}\right)^{n+1}\ddr z\underset{x\rightarrow \infty}{\longrightarrow} \frac{1}{\Psi((n+1)\theta_2)}$$
		which allows us to conclude. We deduce from \eqref{bounds} and the convergence of the series in Corollary \ref{p9} that for large enough $x$, the series $\sum_{n\geq 1}\lambda^n W_n(x)$ is convergent for any $\lambda\geq 0$. Moreover, for any fixed $n$,
		$\lim_{x\rightarrow\infty}\frac{W_n(x)}{(W_1(x))^n}=\prod_{j=1}^n \frac{\Psi(\theta_2)}{\Psi(j\theta_2)}$.
		Recall that for $b\ge 0$, $m(b)=W_1(b)$. Then by replacing $\lambda$ by $\lambda/W_1(b)$ in (\ref{Wninf}) in Theorem
		\ref{c7}, we get the desired limit.
	\end{proof}
	%
	\begin{corol}\label{cvlawcriticalstablerv} Suppose that there exists $1<\alpha\leq 2$  such that for $c>0$, $\Psi(\lambda)\sim c\lambda^\alpha$ as $\lambda\rightarrow 0+$. If  $R(x)=x^{\theta}\ell(x)$ for $x>0$ and $\theta>\alpha$ and $\ell$ is a slowly varying function at $\infty$, then \eqref{cdicondition} is satisfied and
		$\lim_{b\rightarrow \infty}\mbb{E}_\infty\left[e^{-\lambda \frac{T_b}{m(b)}}\right]=\mathbb{E}[e^{-\lambda S_{\alpha,\theta}}]$
		where
		\[\mathbb{E}[e^{-\lambda S_{\alpha,\theta}}]=\left[\sum^\infty_{n=0}\left( \frac{\lambda \Gamma(\theta)}{\Gamma(\theta-\alpha)}\right)^n\prod_{i=1}^{n}\frac{\Gamma(i\theta-i\alpha)}{\Gamma(i\theta-(i-1)\alpha)} \right]^{-1}\!\!\!\!\! \text{ for any } \lambda\geq 0.\]
	\end{corol}
	\begin{proof}  By the assumption, since $\theta>\alpha$ and $\Psi(\lambda)\sim c\lambda^{\alpha}$ when $\lambda$ goes to $0$, we have $\int^{\infty}\frac{\ddr x}{x^{\theta+1}\ell(x)\Psi(1/x)}<\infty$ and \eqref{cdicondition} holds. We establish now that the series $\sum_{n\geq 1}a_n\lambda^{n}$ with $a_n:=\prod_{i=1}^{n}\frac{\Gamma(i\theta-i\alpha)}{\Gamma(i\theta-(i-1)\alpha)}$ converges for all $\lambda\geq 0$. Since $\alpha>1$, then $\Gamma(i\theta-(i-1)\alpha)\geq \Gamma(i\theta-i\alpha)(i\theta-i\alpha)$
		and plainly $a_n\leq \frac{1}{(\theta-\alpha)^{n}n!}$ which ensures the convergence of the series for any $\lambda$.\\
		
		Notice that
		\begin{equation}\label{W1stablelike}
		\begin{split}
		W_1(x)&=\int_{x}^{\infty}\frac{W(z-x)}{R(z)}\ddr z
		=\int_{0}^{\infty}\frac{W(z)}{(x+z)^{\theta}\ell(z+x)}\ddr z\\
		&=\frac{1}{\Gamma(\theta)}\int_{0}^{\infty}\frac{W(z)}{\ell(z+x)}\ddr z\int_{0}^{\infty}\lambda^{\theta-1}e^{-\lambda(z+x)}\ddr \lambda\\
		&=\frac{1}{\Gamma(\theta)}\int_{0}^{\infty}\ddr \lambda e^{-\lambda x}\lambda^{\theta-1}\frac{1}{\ell(x)}\int_{0}^{\infty}W(z)e^{-\lambda z}\frac{\ell(x)}{\ell(x+z)}\ddr z.
		\end{split}
		\end{equation}
		For any fixed $z$, $\frac{\ell(x)}{\ell(x+z)}\underset{x\rightarrow \infty}{\longrightarrow} 1$. By Potter's theorem, for any $C>1$ and $x$ large enough, $\frac{\ell(x)}{\ell(x+z)}\leq C(1+z/x) \leq C(1+z)$.
		Since $\int_{0}^{\infty}W(z)e^{-\lambda z}(1+z)\ddr z<\infty,$  by Lebesgue's theorem we have
		\begin{equation}\label{limitpsi}
		\int_{0}^{\infty}W(z)e^{-\lambda z}\frac{\ell(x)}{\ell(x+z)}\ddr z\underset{x\rightarrow \infty}{\longrightarrow} \frac{1}{\Psi(\lambda)}.
		\end{equation}
		One thus concludes that
		\begin{equation}\label{equivmstable}
		W_1(x)\underset{x\rightarrow \infty}{\sim}\frac{1}{\Gamma(\theta)\ell(x)}\int_{0}^{\infty}\ddr \lambda e^{-\lambda x}\frac{\lambda^{\theta-1}}{\Psi(\lambda)}\underset{x\rightarrow \infty}{\sim} \frac{c\Gamma(\theta-\alpha)}{\Gamma(\theta)}\frac{1}{x^{\theta-\alpha}\ell(x)}.
		\end{equation}
		We are going to prove by induction that for any $n\geq 1$
		\begin{equation}\label{equivWn}
		W_n(x)\sim \frac{x^{n\alpha-n\theta}}{\ell(x)^{n}}\prod_{i=1}^{n}\frac{\Gamma(i\theta-i\alpha)}{\Gamma(i\theta-(i-1)\alpha)} \qquad \text{as} \,\, x\to\infty.
		\end{equation}
		If it holds for $n=m$, then
		\begin{equation*}
		\begin{split}
		W_{m+1}(x)
		&=
		\int^\infty_x \frac{W(z-x)}{z^\theta \ell(z)}W_m(z)\ddr z\\
		&\sim
		\frac{1}{\ell(x)^m}\prod_{i=1}^{m}\frac{\Gamma(i\theta-i\alpha)}{\Gamma(i\theta-(i-1)\alpha)}\int^\infty_x \frac{W(z-x)}{z^\theta \ell(z)}z^{m\alpha-m\theta}\ddr z.
		\end{split}
		\end{equation*}
		Focussing on the integral term above, one has \begin{align*}
		&\int^\infty_x \frac{W(z-x)}{z^\theta \ell(z)}z^{m\alpha-m\theta}\ddr z\\
		&=\frac{1}{\ell(x)}\int_{0}^{\infty}\frac{\ell(x)}{\ell(z+x)}\frac{W(z)}{(z+x)^{(m+1)\theta-m\alpha}}\ddr z\\
		&=\frac{1}{\ell(x)}\frac{1}{\Gamma((m+1)\theta-m\alpha)}\int_{0}^{\infty}\ddr \lambda e^{-\lambda x}\lambda^{(m+1)\theta-m\alpha-1}\int_{0}^{\infty}W(z)e^{-\lambda z}\frac{\ell(x)}{\ell(x+z)}\ddr z\\
		&\sim \frac{1}{\ell(x)}\frac{1}{\Gamma((m+1)\theta-m\alpha)}x^{(m+1)( \alpha-\theta)}\int^\infty_0e^{-\lambda }\lambda^{(m+1)(\theta-\alpha)-1}\ddr \lambda\\
		&=\frac{1}{\ell(x)}\frac{\Gamma((m+1)\theta-(m+1)\alpha)}{\Gamma((m+1)\theta-m\alpha)}x^{(m+1)(\alpha-\theta)},
		\end{align*}
		where the equivalence above follows from (\ref{limitpsi}). We then conclude that $$
		W_{m+1}(x)\sim \frac{x^{(m+1)\alpha-(m+1)\theta}}{\ell(x)^{m+1}}\prod_{i=1}^{m+1}\frac{\Gamma(i\theta-i\alpha)}{\Gamma(i\theta-(i-1)\alpha)} \qquad
		\text{as} \,\, x\to\infty$$
		and (\ref{equivWn}) holds for any $n\geq 1$. The equivalence \eqref{equivWn} and the convergence of the series $\sum_{n\geq 1}a_n\lambda^{n}$ for any $\lambda\geq 0$ entail that for large enough $x$, $\sum_{n\geq 1}\lambda^{n}W_n(x)<\infty$ for any $\lambda\geq 0$. Finally, by (\ref{equivmstable}) and  (\ref{equivWn}) we observe that
		\[\lim_{x\rightarrow \infty}\frac{W_n(x)}{(W_1(x))^n}=\left( \frac{\Gamma(\theta)}{\Gamma(\theta-\alpha)}\right)^n\prod_{i=1}^{n}\frac{\Gamma(i\theta-i\alpha)}{\Gamma(i\theta-(i-1)\alpha)} .\]
		Recall that $W_1(b)=m(b).$ Applying Theorem \ref{c7} and letting $x\to\infty$ give the desired result. \end{proof}
	\begin{remark}\label{exactstable} If $\Psi(\lambda)=c\lambda^{\alpha}$ for all $\lambda \geq 0$ and some $c>0$ and $R(x)=x^{\theta}$ for all $x\geq 0$ with $\theta>\alpha$, we see by replacing $\ell$ by  constant function $1$ in \eqref{W1stablelike} that $m(b)=c\frac{\Gamma(\theta-\alpha)}{\Gamma(\theta)}b^{\alpha-\theta}$ for any $b>0$.
		%
	\end{remark}
	Finding a more general condition over $\Psi$ and $R$ for $\left(\frac{T_b}{\mathbb{E}_\infty(T_b)},b\geq 0\right)$ to converge in law does not seem to follow directly from our approach. We now look for conditions entailing convergence in probability.
	
	\subsection{Convergence in probability}
	We first show that under $\mathbb{H}_1$, $\mathbb{E}_\infty(T_b)\underset{b\rightarrow \infty}{\sim} \varphi(b)$. Recall that $\varphi(b)$ is the first passage time below $b$ of the deterministic flow $(x_t, t \geq 0)$ started from infinity.
	\begin{lemma}\label{equivmphi} Assume $\gamma>0$ and \eqref{cdicondition}. If $\mathbb{H}_1$ is satisfied, then $\mathbb{E}_\infty[T_b]\underset{b\rightarrow \infty}{\sim} \varphi(b)$.
	\end{lemma}
	\begin{proof} By Lemma \ref{l4},
		$\frac{\mathbb{E}_\infty[T_b]}{\varphi(b)}=\frac{\int_{b}^{\infty}\frac{W(x-b)}{R(x)}\ddr x}{\int_{b}^{\infty}\frac{1}{\gamma}\frac{\ddr x}{R(x)}}=1-\frac{\int_{b}^{\infty}\frac{W(\infty)-W(x-b)}{R(x)}\ddr x}{\int_{b}^{\infty}\frac{1}{\gamma}\frac{\ddr x}{R(x)}}$
		and for $h>1$
		\begin{align*}
		0\leq\frac{\int_{b}^{\infty}\frac{W(\infty)-W(x-b)}{R(x)}\ddr x}{\int_{b}^{\infty}\frac{1}{\gamma}\frac{\ddr x}{R(x)}}&=\frac{\int_{bh}^{\infty}\frac{W(\infty)-W(x-b)}{R(x)}\ddr x+\int_{b}^{bh}\frac{W(\infty)-W(x-b)}{R(x)}\ddr x}{\int_{b}^{\infty}\frac{1}{\gamma}\frac{\ddr x}{R(x)}}\\
		&\leq \gamma \big(W(\infty)-W((h-1)b)\big)+\frac{\varphi(b)-\varphi(bh)}{\varphi(b)}.
		\end{align*}
		Therefore,
		$\underset{b\rightarrow \infty}{\limsup} \frac{\int_{b}^{\infty}\frac{W(\infty)-W(x-b)}{R(x)}\ddr x}{\int_{b}^{\infty}\frac{1}{\gamma}\frac{\ddr x}{R(x)}}\leq 1-\underset{b\rightarrow \infty}{\liminf}\frac{\varphi(bh)}{\varphi(b)}$.
		By $\mathbb{H}_1$, we have\\
		$\underset{h\rightarrow 1+}{\limsup}\ \underset{b\rightarrow \infty}{\liminf}\frac{\varphi(bh)}{\varphi(b)}=1$, and we get $\underset{b\rightarrow \infty}{\limsup} \frac{\int_{b}^{\infty}\frac{W(\infty)-W(x-b)}{R(x)}\ddr x}{\int_{b}^{\infty}\frac{1}{\gamma}\frac{\ddr x}{R(x)}}=0.$
	\end{proof}
	
	\begin{theorem}\label{cvinproba}  
		Assume $\gamma>0$ and \eqref{cdicondition}. If $\mathbb{H}_1$ and $\mathbb{H}_3$ are satisfied, then
		$\frac{T_b}{\mathbb{E}_\infty[T_b]}\underset{b\rightarrow \infty}{\longrightarrow 1} \text{ in } \mathbb{P}_\infty \text{-probability}.$
	\end{theorem}
	\begin{proof}  The proof is based on a Chebyshev's inequality type argument. By Corollary \ref{c4}, for any $b>0$ and $h>1$
		\begin{align*}
		&\text{Var}_\infty(T_b)=2\int_{b}^{\infty}\frac{W(x-b)}{R(x)}\ddr x\int_{x}^{\infty}\frac{W(y-b)-W(y-x)}{R(y)}\ddr y\\
		&\leq 2\int_{b}^{\infty}\frac{\ddr x}{\gamma R(x)}\int_{x+(h-1) b}^{\infty}\frac{W(y-b)-W(y-x)}{R(y)}\ddr y\\
		&\quad+2\int_{b}^{\infty}\frac{\ddr x}{\gamma R(x)}\int_{x}^{x+(h-1) b}\frac{W(y-b)-W(y-x)}{R(y)}\ddr y
		=: J_1+J_2,
		\end{align*}
		where
		\begin{align*} J_1&\leq 2\gamma \int_{b}^{\infty}\frac{\ddr x}{\gamma R(x)}\int_{x+(h-1) b}^{\infty}\frac{\ddr y}{\gamma R(y)}\left(W(\infty)-W((h-1) b)\right)\\
		&\leq 2\gamma\big(W(\infty)-W((h-1) b)\big)\varphi(b)\varphi(bh)
		\leq 2\gamma\big(W(\infty)-W((h-1) b)\big)\varphi(b)^2
		\end{align*}
		and, since $W(y-b)-W(y-x)\leq W(\infty)\leq \frac{1}{\gamma}$, then
		\begin{align}\label{I2} J_2\leq &2\int_{b}^{\infty}\frac{\varphi(x)-\varphi(x+(h-1) b)}{\gamma R(x)}\ddr x \nonumber\\
		=&2\left.\bigg[\!\!-\varphi(x)\big(\varphi(x)-\varphi(x+(h-1) b)\big)\bigg]_{x=b}^{x=\infty}\!\!\!\!\!\!+2\int_{b}^{\infty}\varphi(x)\left(\frac{1}{R(x+(h-1) b)}-\frac{1}{R(x)}\right)\ddr x.\right.\nonumber
		\end{align}
		Recall $V$ defined in \eqref{nobigvalley}. The latter integral is bounded above by
		$2V(b,h-1)\int_{b}^{\infty}\frac{\varphi(x)}{R(x)}\ddr x$. Therefore,  $J_2\leq 2 \varphi(b)[\varphi(b)-\varphi(bh)]+ 2\gamma \varphi(b)^2V(b,h-1)$ and for all $b>0$ and $h>1$,
		\begin{equation}\label{controlvar}
		\frac{\text{Var}_\infty(T_b)}{\varphi(b)^2}\leq 2\gamma\big(W(\infty)-W((h-1) b)\big)+2\frac{\varphi(b)[\varphi(b)-\varphi(bh)]}{\varphi(b)^2}+2\gamma V(b,h-1)
		\end{equation}
		
		One further has  $\underset{b\rightarrow \infty}{\limsup} \frac{\text{Var}_\infty(T_b)}{\varphi(b)^2}\leq 2\left(1-\underset{b\rightarrow \infty}{\liminf}\frac{\varphi(bh)}{\varphi(b)}\right)+{ 2 \gamma\underset{b\rightarrow \infty}{\limsup} V(b,h-1)}$
		and letting $h$ go to $1$, we get by $\mathbb{H}_1$ and $\mathbb{H}_3$ that $\underset{b\rightarrow \infty}{\limsup} \frac{\text{Var}_\infty(T_b)}{\varphi(b)^2}=0.$
	\end{proof}
	
	\begin{remark}\label{remarkcontrolvar}
		The upper bound \eqref{controlvar} for the variance of $T_b/ \varphi(b)$ under $\mathbb {P}_\infty $ has three terms, the first controls the random fluctuations, the second 	 the speed of coming down from infinity of the deterministic flow and the third controls the valleys depth in the neighbourhoods of infinity.
	\end{remark}
	\subsection{Almost-sure convergence}
	We now look for conditions ensuring that $(T_b/\mathbb{E}_\infty[T_b], b\geq 0)$ converges almost-surely towards $1$ under $\mbb{P}_\infty$. Recall $\varphi$ defined in (\ref{phi}), $V$ defined in (\ref{nobigvalley}), $W$ the scale function and for any $z>0$, set $\Delta(z):=W(\infty)-W(z)$.
	\begin{theorem}\label{almost_sure_H3} Assume \eqref{cdicondition} and $\gamma>0$. If $R$ satisfies the following condition:
		
		\noindent $\mathbb{H}_4$:
		there is a map $p:$  $\mbb{R}_+\to\mbb{R}_+$ such that $z\mapsto zp(z)$ is non-decreasing, $p(z)\underset{z\rightarrow \infty}{\longrightarrow} 0$ and
		\begin{itemize}
			\item[(i)] as $z$ goes to $\infty$, $\varphi(z)-\varphi(z+zp(z))=O\left(\varphi(z)\Delta(zp(z))\right),$
			\item[(ii)] there is $c>1$ such that
			$\int^{\infty}\frac{\Delta(zp(z)) }{\varphi(cz)R(cz)}\ddr z<\infty,$
			\item[(iii)] as $z$ goes to $\infty$, $V(z,p(z))=O\left(\Delta(zp(z))\right)$,
		\end{itemize}
		then
		$\frac{T_b}{\mbb{E}_\infty[T_b]}\underset{b\rightarrow \infty}{\longrightarrow}  1 \quad {\mbb P}_\infty\text{-a.s.}$
	\end{theorem}
	We will give examples of functions satisfying the assumptions of Theorem \ref{almost_sure_H3} later in Corollary \ref{stablelikepsiandcramer}. The proof of Theorem \ref{almost_sure_H3} is based on a strong law of large numbers for tails of random series, established in Klesov \cite{MR704328} and Nam \cite{Nam}. Recall that for any $z>0$, $m(z)=\mathbb{E}_\infty[T_z]$.
	\begin{lemma}\label{taillemma} Let $(z_n)_{n\geq 1}$ be a positive non-decreasing sequence going to $\infty$. If \begin{equation}\label{series}
		\sum_{n\geq 1}\frac{\text{Var}_{\infty}(T_{z_n})}{m(z_n)^2}<\infty,
		\end{equation}
		then
		\begin{equation}\label{astail} \frac{T_{z_n}}{m(z_n)}\underset{n\rightarrow \infty}{\longrightarrow} 1, \ \mathbb{P}_\infty-\text{a.s}.
		\end{equation}
	\end{lemma}
	\begin{proof}
		Since the process has no negative jumps, under $\mbb{P}_\infty$, for any $n\geq 1$, $T_{z_n}=\sum_{k\geq n}\tau_k$ with $\tau_k:=T_{z_{k}}-T_{z_{k+1}}$ and by the strong Markov property of the process $(X_t,t\geq 0)$, the random variables $(\tau_k, k\geq 1)$ are independent with the same distribution as  $T_{z_{k+1}}\circ\theta_{T_{z_{k}}}$, where $\theta$ is the shift operator (see e.g. \cite[p.146]{Kallenberg}). Under $\mbb{P}_\infty$, set for any $n\geq 1$,
		$\zeta_n:=T_{z_n}-m(z_n)=\sum_{k\geq n}(\tau_k-\mathbb{E}_\infty[\tau_k]).$
		Applying \cite[Proposition 1]{MR704328} with, in the notation there, $t=2$, $b_k=1/m(z_k)$ and $\xi_k=\tau_k-\mathbb{E}_\infty[\tau_k]$,
		if $\sum_{n=1}^{\infty}\frac{\text{Var}_\infty(\tau_n)}{m(z_{n})^2}<\infty$, then
		$\frac{\zeta_n}{m(z_n)}=\frac{T_{z_n}}{m(z_n)}-1\underset{n\rightarrow \infty}{\longrightarrow} 0 \ \ \mathbb{P}_\infty$-a.s.
		Since $T_{z_{n+1}}+\tau_n=T_{z_{n}}$ and $\tau_n$ is independent of $T_{z_{n+1}}$,  $\text{Var}_\infty(\tau_{n})=\text{Var}_\infty(T_{z_n})-\text{Var}_\infty(T_{z_{n+1}})\leq \text{Var}_\infty(T_{z_n})$ and
		$\frac{\text{Var}_\infty(\tau_n)}{m(z_{n})^2}\leq \frac{\text{Var}_\infty(T_{z_n})}{m(z_{n})^2}$. This concludes the proof.
	\end{proof}
	We now deal with the proof of Theorem \ref{almost_sure_H3}.
	\begin{proof}[Proof of Theorem \ref{almost_sure_H3}]
		We start by finding a sequence $(z_n,n\geq 1)$ such that under the assumption $\mathbb{H}_4$, the series \eqref{series} converges.
		Recall from Lemma \ref{equivmphi} that under $\mathbb{H}_1$, $m(z)\sim \varphi(z)$ as $z$ goes to $\infty$. One thus only needs to study $\sum_{n\geq 1}\frac{\text{Var}_{\infty}(T_{z_n})}{\varphi(z_n)^2}$. Condition $\mathbb{H}_4$ enables us to control precisely the variance of $T_{z_n}$ under $\mathbb{P}_\infty$ for large $n$. Applying  \eqref{controlvar}
		with $h=1+p(z_n)$ (so that $(h-1)z_n=z_np(z_n)$) provides
		\begin{equation}\label{boundforzn}\frac{\text{Var}_\infty(T_{z_n})}{ 2\varphi(z_n)^2}\leq \gamma \Delta(z_np(z_n))+\frac{\varphi(z_n)-\varphi(z_n(1+p(z_n)))}{\varphi(z_n)}+\gamma V(z_n,p(z_n)).
		\end{equation}
		By the conditions (i) and (iii) in $\mathbb{H}_4$,
		$\frac{\text{Var}_\infty(T_{z_n})}{2\varphi(z_n)^2}\leq C\Delta(z_np(z_n))$
		for some constant $C>0$. It is therefore sufficient to identify a sequence $(z_n,n\geq 1)$ such that under the assumption $\mathbb{H}_4$-(ii), $\sum_{n=1}^{\infty}\Delta(z_np(z_n))<\infty$. For any fixed constant $c>1$, and any $c_1<\frac{1}{\gamma}$, for $z$ large enough
		$$m(z)=\int_{z}^{\infty}\frac{W(x-z)}{R(x)}\ddr x\geq W((c-1)z)\int_{cz}^{\infty}\frac{\ddr x}{R(x)}\geq c_1 \gamma\varphi(cz).$$
		Note that $z\mapsto m(z)$ is continuous and strictly decreasing in $z$, and $m(z)\rightarrow 0+$ as $z\rightarrow\infty$. Given $0<\rho<1$ and $z_0$ large enough, for $n\geq 1$  recursively define
		\begin{equation} \label{increasez_n} z_n:=\inf\{z>z_{n-1}: m(z)=(1-\rho)m(z_{n-1})\}.
		\end{equation}
		Then by continuity $m(z_n)=(1-\rho)^n m(z_0)$. Note that $m(z_n)\geq c_1 \gamma \varphi(cz_n)$ for all $n\geq0$. Then
		$z_n\geq y_n:=\frac{1}{c}\varphi^{-1}\left(c_2(1-\rho)^{n}\right), \,\,\, n=1,2, \dots $
		with $c_2=\frac{m(z_0)}{\gamma c_1}$. Since $z\mapsto zp(z)$ is non-decreasing and $z\mapsto \Delta(z)$ is non-increasing,  for any $n$ we have $\Delta(z_np(z_n))\leq \Delta(y_np(y_n))$.
		
		We now show  that $\sum_{n=1}^{\infty}\Delta(y_np(y_n))<\infty$. This will follow by comparison with an integral. set $\beta:=-\log(1-\rho)$. For any $z\in ]n,n+1[$, we have $y_{n+1}\geq \frac{1}{c}\varphi^{-1}(c_2e^{-\beta z})=:u(z)$. Therefore,
		if
		\begin{equation}\label{integral1}
		\int^{\infty}\Delta(u(z)p(u(z)))\ddr z<\infty,
		\end{equation}
		then $\sum_{n=1}^{\infty}\Delta(y_np(y_n))<\infty$. By changing of variable and setting $v:=u(z)$, one can check that $\ddr z=\frac{c}{\beta}\frac{\ddr v}{\varphi(cv)R(cv)}$ and (\ref{integral1}) holds if and only if
		$\int^{\infty}\frac{\Delta(vp(v))}{\varphi(cv)R(cv)}\ddr v<\infty$
		which is  $\mathbb{H}_4$-(ii). Lemma \ref{taillemma} entails the almost-sure convergence \eqref{astail} along the sequence \eqref{increasez_n}.
		
		To conclude, we now show the almost-sure convergence along any sequence. For any $z>z_0$ such that $z\in [z_{n-1},z_n[$ we have
		$$\frac{T_{z_{n-1}}}{m(z_{n-1})}\geq \frac{T_{z}}{m(z_{n-1})}\geq (1-\rho)\frac{T_z}{m(z)}\geq (1-\rho)\frac{T_{z_n}}{m(z_{n-1})}=(1-\rho)^2\frac{T_{z_n}}{m(z_{n})}.$$
		Therefore,
		$1-\rho\leq \underset{z\rightarrow \infty}{\liminf} \frac{T_z}{m(z)}\leq \underset{z\rightarrow \infty}{\limsup} \frac{T_z}{m(z)}\leq \frac{1}{1-\rho}$.
		Since $\rho$ can be arbitrarily close to $0$, the almost-sure convergence to $1$ is established.
	\end{proof}
	\begin{remark}\label{alternativecond}
		The condition $\mathbb{H}_4$-(iii) requires somehow that the fluctuations of the L\'evy process prevail over those of the function $R$. We provide here an alternative condition which does not involve the function $\Delta$ but only $R$. Set the condition $\mathbb{H}_4$-$\text{(iii')}$:
		
		There exists a decreasing function $\bar{V}$ and $c>1$ such that
		$V(z,p(z))\leq \bar{V}(z)$  for large enough  $z$ and $ \int^{\infty}\frac{\bar{V}(v)}{\varphi(cv)R(cv)}\ddr v<\infty$.
		By comparing the series $\sum_{n\geq 1} V(z_n,p(z_n))$ with the integral above, we check  that under $\mathbb{H}_4$-$\text{(iii')}$, $\sum_{n\geq 1} V(z_n,p(z_n))<\infty$. By repeating the arguments of the proof, from \eqref{boundforzn}, we get that  $\mathbb{H}_4$-(iii) can be replaced by $\mathbb{H}_4$-(iii') in Theorem \ref{almost_sure_H3}.
	\end{remark}
	Condition $\mathbb{H}_4$ is difficult to verify in general. The next lemma further simplifies  $\mathbb{H}_4$-(i) and $\mathbb{H}_4$-(ii)  for $R$  regularly varying.
	\begin{lemma}\label{regularlyvaryingR} Assume  that $R$ is regularly varying with index $\theta>1$. If there exists a map $p:\mbb{R}_+\to\mbb{R}_+$ such that $z\mapsto zp(z)$ is non-decreasing, $p(z)\underset{z\rightarrow \infty}{\longrightarrow} 0$, $\frac{p(z)}{\Delta(zp(z))}$ is bounded for large enough $z$ and $\int^{\infty}\frac{\Delta(zp(z))}{z}\ddr z<\infty$, then both $\mathbb{H}_4$-(i) and $\mathbb{H}_4$-(ii)  hold.
	\end{lemma}
	\begin{proof}
		A simple application of the mean value theorem provides the following useful bound for condition $\mathbb{H}_4$-(i):
		\begin{equation}\label{bound-dp}
		d_p(z):=\frac{1}{\Delta(zp(z))}\left(\frac{\varphi(z)-\varphi(z+zp(z))}{\varphi(z)}\right)\!\leq \frac{zp(z)}{\Delta(zp(z))\varphi(z)}\sup_{u\in ]z,z(1+p(z))[}\frac{1}{R(u)}.
		\end{equation}
		Since $R$ is regularly varying at $\infty$ with index $\theta$,  $R$ takes the form $R(x)=x^{\theta}\ell(x)$, where $\ell$ is a slowly varying function at $\infty$. By (\ref{bound-dp}), one has
		\begin{align*}
		d_p(z)
		&\leq \frac{zp(z)}{\Delta(zp(z))}\frac{1}{\varphi(z)R(z)}\sup_{u\in ]z,z(1+p(z))[}\frac{\ell(z)}{\ell(u)}.
		\end{align*}
		By the uniform convergence theorem for slowly varying function, see e.g. \cite[Theorem 1.2.1]{MR1015093}, we have $\sup_{u\in ]z,z(1+p(z))[}\frac{\ell(z)}{\ell(u)}\underset{z\rightarrow \infty}{\longrightarrow} 1$. Moreover, $\varphi(z)R(z)\underset{z\rightarrow \infty}{\sim} (\theta-1)z$. Therefore, the map $d_p$ is bounded as soon as $z\mapsto \frac{p(z)}{\Delta(zp(z))}$ is bounded. The condition $\mathbb{H}_4$-(ii) is readily equivalent to  $\int^{\infty}\frac{\Delta(zp(z))}{z}\ddr z<\infty$.
	\end{proof}
	
	We now apply Lemma \ref{regularlyvaryingR} and Theorem \ref{almost_sure_H3} to different branching mechanisms (those in the setting of Theorem \ref{regularvarR}). Recall that we denote by $\pi$ the L\'evy measure.
	\begin{corol}\label{stablelikepsiandcramer}
		Assume $\gamma>0$ and $R$ regularly varying at $\infty$ with index $\theta>1$.  Then \eqref{cdicondition} is satisfied. Set the conditions
		\begin{enumerate}
			\item[(a)] $\Psi(\lambda)-\gamma \lambda\sim c\lambda^{1+\delta} \,\,\, \text{ as}\,\,\, \lambda\rightarrow 0+$ for some $\delta\in (0,1)$, $c>0$ and\\
			\begin{equation}\label{stablevalleys}V(z,\frac{1}{\sqrt{z}}):=\sup_{u\geq z}\left(\frac{R(u)}{R(u+\sqrt{z})}-1\right)_+=O(z^{-\delta/2})
			\quad\mbox{as}~z\to\infty;
			\end{equation}
			\item[(b)] There exists $\nu\in( 0,\infty)$ such that $\Psi(-\nu)=0$, $\int^1_0u\pi(\ddr u)<\infty$ and
			\begin{equation}\label{cramervalleys}V(z,\frac{2 \ln \ln z}{\nu z}):=\sup_{u\geq z}\left(\frac{R(u)}{R(u+\frac{2}{\nu}\ln \ln(z))}-1\right)_+=O(\ln(z)^{-2})\quad\mbox{as}~z\to\infty.
			\end{equation}
		\end{enumerate}
		If either condition (a) or condition (b) is satisfied, then
		$\mathbb{P}_\infty$-almost-surely,
		$\frac{T_b}{\mathbb{E}_\infty[T_b]}\underset{b\rightarrow \infty}{\longrightarrow} 1$.
	\end{corol}
	\begin{proof}
		By Karamata's theorem, see e.g. \cite[Proposition 1.5.10]{MR1015093},  \eqref{cdicondition} is satisfied. We next verify the conditions of Lemma \ref{regularlyvaryingR}. Assume first that condition \textit{(a)} holds. By Lemma \ref{l1},
		\begin{equation*}
		\begin{split}
		&\int_0^\infty (W(\infty)-W(x))e^{-\lambda x}\ddr x=\frac{W(\infty)}{\lambda}-\frac{1}{\Psi(\lambda)}  \\
		&\sim \frac{(\lambda/W(\infty)+c\lambda^{1+\delta})W(\infty)-\lambda}{\lambda(\lambda/W(\infty)+c\lambda^{1+\delta})}
		\sim W(\infty)^2c\lambda^{\delta-1} \quad\text{as}\quad \lambda\rightarrow 0.
		\end{split}
		\end{equation*}
		By the Tauberian theorem, see e.g. \cite[p.10]{MR1406564}, we have
		$\Delta(x)=W(\infty)-W(x)\sim c_1x^{-\delta}$ as $ x\rightarrow\infty$
		with $c_1=\frac{cW(\infty)^2}{\Gamma(2-\delta)}$.
		Then $\Delta(zp(z))\underset{z \rightarrow \infty}{\sim} c_1 (zp(z))^{-\delta}$ and $\frac{p(z)}{\Delta(zp(z))}\underset{z\rightarrow \infty}{\sim} \frac{1}{c_1}z^{\delta}p(z)^{ 1+\delta}$.
		Choose $p(z)=\frac{1}{\sqrt{z}}$. Then $z\mapsto zp(z)$ is non-decreasing, and $\frac{p(z)}{\Delta(zp(z))}\underset{z\rightarrow \infty}{\sim} \frac{1}{c_1}z^{\delta/2-1/2}$ which is bounded. Since $\int^{\infty}\frac{\ddr z}{z^{\delta/2+1}}<\infty$, we have that  $\int^{\infty}\frac{\Delta(zp(z))}{z}\ddr z<\infty$ and we conclude by Lemma \ref{regularlyvaryingR} that $\mathbb{H}_4$-(i) and $\mathbb{H}_4$-(ii) hold true. Condition \eqref{stablevalleys} corresponds to $\mathbb{H}_4$-(iii) for $p(z)=z^{-1/2}$. Since $R$ is regularly varying with index $\theta>1$, one can readily check that $\varphi$ satisfies $\mathbb{H}_1$. Theorem \ref{almost_sure_H3} thus applies.\\
		
		Assume now that condition \textrm{(b)} holds. Recall from Lemma \ref{l1} that for any $x\geq 0$, $\mathbb{P}_0(\tau_x^{+}<\infty)=\frac{W(\infty)-W(x)}{W(\infty)}$.	
		By Cram\'er's theorem, see e.g. \cite[Theorem 7.6]{MR3155252}, we have
		$$\underset{x\rightarrow \infty}{\lim} e^{\nu x}\mathbb{P}_0(\tau_x^{+}<\infty)=\frac{\Psi'(0+)}{\Psi'(-\nu)}=:c_\nu\in [0,\infty)$$
		with $c_\nu=0$ if $\Psi'(-\nu)=\infty$. We can check from \eqref{LK} that the assumption $\int_0^1 u\pi(\ddr u)<\infty$ entails $\Psi'(-\nu)<\infty$. Therefore, $c_\nu>0$ and $W(\infty)-W(x)\sim \frac{c_\nu}{\gamma} e^{-\nu x}$
		as $x$ goes to $\infty$. Choose $p(z)=\frac{2\ln \ln z}{\nu z}$ for large $z$. The map $z\mapsto zp(z)$ is non-decreasing, $p(z)\underset{z\rightarrow \infty}{\longrightarrow} 0$ and moreover $\Delta(zp(z))\sim \frac{c_\nu}{\gamma\ln(z)^{2}}$ as $z$ goes to $\infty$. Therefore $\frac{p(z)}{\Delta(zp(z))}\sim \frac{2}{\gamma\nu  c_\nu}\frac{\ln\ln z}{z}\ln(z)^{2}$ as $z$ goes to $\infty$ and $z\mapsto \frac{p(z)}{\Delta(zp(z))}$ is bounded for $z$ large enough. Note that  $\int^{\infty}\frac{\ddr z}{z\ln(z)^2 }<\infty$ and then
		$\int^{\infty}\frac{\Delta(zp(z))}{z}\ddr z<\infty$. By Lemma \ref{regularlyvaryingR}, $\mathbb{H}_4$-(i) and $\mathbb{H}_4$-(ii) are satisfied. Condition \eqref{cramervalleys} corresponds to $\mathbb{H}_4$-(iii) with $p(z)=\frac{2\ln \ln z}{\nu z}$ and finally Theorem \ref{almost_sure_H3} applies.
	\end{proof}
	We close this section by providing an example of a regularly varying branching rate $R$ with valleys in any neighborhood of $\infty$.
	\begin{example}\label{examplenonincreas} Let $\theta>1$, $0<v<1$ and $x_0>0$. Assume that $R(x)=x^{\theta}\left(2+\frac{\cos x}{x^{v}}\right)$ for all $x\geq x_0$. The function $R$ is not monotonic on $(b,\infty)$ for any $b>x_0$ and simple calculations provide that for any $\rho>0$ and large enough $z$, $V(z,\rho)\leq C\rho z^{1-v}$ for some constant $C>0$. One can check that \eqref{cramervalleys} is always fulfilled. Moreover, $V(z,1/\sqrt z)\leq C z^{1/2-v}$ and \eqref{stablevalleys} holds when $v\geq \frac{1+\delta}{2}$.
		
		As noticed  in the proof of Corollary \ref{stablelikepsiandcramer}, $\mathbb{H}_4$-(i) and $\mathbb{H}_4$-(ii) are verified and \eqref{stablevalleys} corresponds to $\mathbb{H}_4$-(iii). Instead of \eqref{stablevalleys}, we may thus verify $\mathbb{H}_4$-(iii') as defined in Remark \ref{alternativecond} . Set $\bar{V}(z):=Cz^{1/2-v}$, when $v>1/2$, $\bar{V}$ is decreasing and $\int^{\infty}\frac{\bar{V}(z)}{z}\ddr z<\infty$. Therefore, when $v>1/2$, $\mathbb{H}_4$-(iii') holds and one concludes by Theorem \ref{almost_sure_H3} that $\mathbb{P}_\infty$-almost-surely, $\frac{T_b}{\mathbb{E}_\infty[T_b]}\underset{b\rightarrow \infty}{\longrightarrow} 1$.
	\end{example}
	\section{Speeds of coming down from infinity}\label{speeds}
	In this final section, we show how to invert the results obtained on $(T_b, b\geq 0)$ in the previous section to study the speed of coming down from infinity. We prove Theorem \ref{regularCDI}, Theorem  \ref{regularvarR} and Proposition \ref{fastCDI}.
	\subsection{Running infimum at small times}
	Recall Theorem \ref{main1}. Under \eqref{cdicondition}, the function $m$ is well-defined, positive continuous and strictly decreasing. We write $m^{-1}$ for its inverse function. Note that
	$\mbb{E}_\infty(T_{m^{-1}(t)})=t $ and $m^{-1}(t)\rightarrow\infty $ as $t\rightarrow 0+$.
	Let
	${\underline X}_t:=\inf_{0\leq s\leq t}X_s, \,\,\, t\geq 0$
	be the running minimum process for $X$. Under $\mbb{P}_\infty$, the asymptotic behavior of $T_b$ for large $b$ is associated to the small time asymptotic behavior for ${\underline X}_t$.
	\begin{lemma}\label{liminf_a}
		Assume that \eqref{cdicondition} holds.
		\begin{itemize}
			\item[(a)]
			Suppose that for any $h>1$,
			$\liminf_{x\rightarrow\infty}\frac{m(x)}{m(hx)}>1$
			and that we have
			\begin{equation}\label{converge}
			T_b/m(b)\rightarrow 1 \text{ in } \mbb{P}_\infty\text{-probability as  } b\rightarrow\infty.
			\end{equation}
			Then under  $\mbb{P}_\infty$,
			$\frac{{\underline X}_t}{m^{-1}(t)}\rightarrow 1 \,\,\,\text{in probability as} \,\,\, t\rightarrow 0+.$
			\item[(b)]
			Suppose that for any $h>1$,
			$\liminf_{x\rightarrow\infty}\frac{m(x)}{m(hx)}=\infty$  
			and that ${T_b}/{m(b)}$ converges in law under $\mbb{P}_\infty$. Then under $\mbb{P}_\infty $,
			$\frac{{\underline X}_t}{m^{-1}(t)}\rightarrow 1$ in probability as $ t\rightarrow 0+$.
		\end{itemize}
	\end{lemma}
	\begin{proof}
		We show (a). By assumption, for any $h>0$, there exists a constant $\underbar{c}_h\in(1,\infty)$ such that $\liminf_{x\rightarrow\infty}\frac{m(x)}{m(hx)}>\underbar{c}_h.$
		For $h>1, 0<\rho<\underbar{c}_h, \delta>0$ and $t>0$ small enough,
		\begin{equation*}
		\begin{split}
		&\mbb{P}_\infty \left(T_{hm^{-1}(t)}\leq t \right)
		=\mbb{P}_\infty \left(T_{hm^{-1}(t)}\leq \mbb{E}_\infty \left[T_{m^{-1}(t)}\right] \right)\\
		&=\mbb{P}_\infty \left(\frac{T_{hm^{-1}(t)}}{\mbb{E}_\infty[T_{hm^{-1}(t)}] }\leq \frac{\mbb{E}_\infty[T_{m^{-1}(t)}]}{\mbb{E}_\infty[T_{hm^{-1}(t)}]} \right)
		\geq\mbb{P}_\infty \left(\frac{T_{hm^{-1}(t)}}{\mbb{E}_\infty[T_{hm^{-1}(t)}] }\leq \underbar{c}_h- \rho \right).
		\end{split}
		\end{equation*}
		Similarly,
		\begin{equation*}
		\begin{split}
		\mbb{P}_\infty \left(T_{m^{-1}(t)/h}\geq t \right)
		&=\mbb{P}_\infty \left(\frac{T_{m^{-1}(t)/h}}{\mbb{E}_\infty[T_{m^{-1}(t)/h}] }\geq \frac{\mbb{E}_\infty[ T_{m^{-1}(t)}]}{\mbb{E}_\infty[T_{m^{-1}(t)/h}]} \right)\\
		&\geq\mbb{P}_\infty \left(\frac{T_{m^{-1}(t)/h}}{\mbb{E}_\infty[T_{m^{-1}(t)/h}] }\geq \frac{1}{\underbar{c}_h}+\rho \right).
		\end{split}
		\end{equation*}
		Then  for $\underbar{c}_h- \rho>1, \, \frac{1}{\underbar{c}_h}+\rho<1$ and for all $\ep>0$ and all  $t$ small enough, by (\ref{converge}) we have
		$\mbb{P}_\infty (T_{hm^{-1}(t)}\leq t, \,\,\,\,\, T_{m^{-1}(t)/h}\geq t )\geq 1-\ep$.
		Observe that ${\underline X}_{T_b}=b $ and ${\underline X}_t$ decreases in $t$. Then
		$\mbb{P}_\infty ( {\underline X}_t\in [m^{-1}(t)/h, m^{-1}(t)h])\geq \mbb{P}_\infty (t\in [T_{m^{-1}(t)h}, T_{m^{-1}(t)/h}]  ) \geq 1-\ep$.
		Since $\ep$  is arbitrary,
		$\lim_{t\rightarrow 0+}\mbb{P}_\infty ( {\underline X}_t\in [m^{-1}(t)/h, m^{-1}(t)h])= 1$.  The desired limit in (a) follows by letting $h\rightarrow 1+$.
		
		The proof for (b) is similar and we leave it to interested readers.
	\end{proof}

	We now  identify conditions under which  $\lim_{t\rightarrow 0+}{\underline X}_t/{X_t}=1 $ \, $\mbb{P}_\infty$-a.s., i.e.
	for $t$ close to $0$, the sample path of  $X_t$ is ``almost'' a decreasing function with  relatively small upward fluctuations.
	
	\begin{prop}\label{equivinf}
		Assume $\gamma>0$ and \eqref{cdicondition}. Then $\mbb{P}_\infty$-a.s.,
		$\lim_{t\rightarrow 0+}\frac{{\underline X}_t}{X_t}=1.$
	\end{prop}
	\begin{proof}
		Recall the first exit times  $\tau^+_b$ and $\tau^-_b$ for the L\'evy process $Z$, see Section \ref{scalefunction}. For any $b>0$, set $T^+_b:=\inf\{t>0: X_t>b\}$ the first exit time of $X$ above $b$.
		Since the process $X$ comes down from infinity, from Lemma \ref{equivcdi}, there exists $b>0$ such that $\mathbb{E}_\infty[T_{b}]<\infty $.
		Observe that, since the process $X$ is a time-change of process $Z$,  for $b<b_1<x<b_2$ and $Y_0=X_0=x $, we have
		$\{T_{b_1}<T^+_{b_2} \}= \{\tau^-_{b_1}<\tau^+_{b_2} \}$
		and
		$\{T_{b_1}>T^+_{b_2} \}= \{\tau^-_{b_1}>\tau^+_{b_2} \}$.
		Then, the fluctuations of $X$ can be studied via  the fluctuations of $Z$.
		\\
		
		Given $a>1$, for any $\delta>0$ and sequence $(a_n)_{n\geq 1}:=(a^{1+n\delta})_{n\geq 1}$,
		choose $m$ large enough so that $a_m>b$ and $W(a_{n+1}-a_{n-1})>W(\infty)/2 $ for all $n\geq m $.
		Further choose $k$ large enough so that $a^{k\delta}(a^{\delta}-1)\geq 1$. Then
		for all $n\geq 1$,
		$$a_{n+k+1}-a_{n+k}=a^{1+n\delta}a^{k\delta}(a^{\delta}-1)\geq a_{n}-a_{n-2},$$
		which implies
		$W(a_{n+k+1}-a_{n+k})\geq W(a_n-a_{n-2})$.
		It follows that
		\begin{equation}\label{BC_lemma}
		\begin{split}
		&\sum_{n=m+1}^\infty \mbb{P}_{a_n}(T^+_{a_{n+1}}<T_{a_{n-1}})=\sum_{n=m+1}^\infty  \left(1-\frac{W(a_{n+1}-a_n)}{W(a_{n+1}-a_{n-1})}\right)\\
		&\leq\frac{2}{W(\infty)}\sum_{n=m+1}^\infty \big(W(a_{n+1}-a_{n-1})- W(a_{n+1}-a_{n})\big)  \\
		&\leq\frac{2}{W(\infty)}\lim_{n'\rightarrow\infty}\left\{\sum_{n=m+1}^{n'-k-1} \big(W(a_{n+1}-a_{n-1})- W(a_{n+k+2}-a_{n+k+1})\big) \right. \\
		&\qquad\qquad\qquad { \left.+\sum_{n=n'-k}^{n'} W(a_{n+1}-a_{n-1})  \right\}   }\\
		&\leq \frac{2}{W(\infty)}\times (k+1)W(\infty)<\infty,
		\end{split}
		\end{equation}
		where we used Lemma \ref{l1} in the first equality.
		\noindent
		Since the process $X$ comes down from infinity, by the strong Markov property and the lack of negative jumps for $X$,
		\[\mbb{P}_\infty (T^+_{a_{n+1}}\circ\theta_{T_{a_{n}}}< T_{a_{n-1}}\circ\theta_{T_{a_{n}}})=\mbb{P}_{a_n}(T^+_{a_{n+1}}<T_{a_{n-1}}),\]
		where $\theta$ is the shift operator, see e.g. \cite[p.146]{Kallenberg}. Applying the Borel-Cantelli lemma, by (\ref{BC_lemma}) we have $\mbb{P}_\infty$-a.s. for all $n$ large enough,
		\begin{equation}\label{fluc}
		T^+_{a_{n+1}}\circ\theta_{T_{a_{n}}} \ge    T_{a_{n-1}}\circ\theta_{T_{a_{n}}},
		\end{equation}
		It follows from (\ref{fluc}) that  $\mbb{P}_\infty$-a.s., for any $t$ small enough,
		$t\in [T_{a_{n}}, T_{a_{n-1}})$ implies that $X_t\in (a_{n-1}, a_{n+1})$ and consequently,
		$ 1\geq\frac{{\underline X}_t}{X_t}\geq \frac{a_{n-1}}{a_{n+1}}=a^{-2\delta}$.
		Since $T_{a_{n}}\rightarrow 0$ as $n\to\infty$ under $\mbb{P}_\infty$ and  $\delta>0$ is arbitrary,
		the limit then follows by letting $\delta\rightarrow 0+$.
	\end{proof}
	The following result finds a condition on $\Psi$ under which  $\lim_{t\rightarrow 0+}{{\underline X}_t}/{X_t}=1$ fails.
	
	\begin{prop}\label{bigexcursionsabove}
		Assume $\Psi(\lambda)\sim c \lambda^{\alpha}$ as $\lambda$ goes to $0$, for some $\alpha\in (1,2]$ and $c>0$. Suppose that $X$ comes down from infinity, then $\mbb{P}_\infty$-a.s.
		$\limsup_{t\rightarrow 0+}\frac{X_t}{{\underline X}_t}=\infty.$
	\end{prop}
	
	\begin{proof}
		\noindent From Lemma \ref{l1} and the Tauberian theorem, one sees that $\Psi(\lambda)\sim c\lambda^{\alpha}$ as $\lambda\rightarrow 0+$ for some constant $c>0$ is equivalent to $W(x)\sim c' x^{\alpha-1}$ as $x\rightarrow\infty$.
		Using Lemma \ref{l1} and the same time-change techniques as in the proof of Proposition~\ref{equivinf}, we see that for $a>1$,
		\begin{equation}\label{p6.4f}
		\begin{split}
		\lim_{n\rightarrow\infty}\mbb{P}_{a^n} (T^+_{a^{n+1}}<T_{a^{n-1}})
		&=\lim_{n\rightarrow\infty}\mbb{P}_{a^n}\left(\tau^+_{a^{n+1}}<\tau^-_{a^{n-1}}\right)
		=1-\lim_{n\rightarrow\infty}\frac{W(a^{n+1}-a^n)}{W(a^{n+1}-a^{n-1})}\\
		&= 1-\lim_{n\rightarrow\infty}\left(\frac{a^{n+1}-a^n}{a^{n+1}-a^{n-1}}\right)^{\alpha-1}
		=1-\left(\frac{a}{a+1}\right)^{\alpha-1}>0.
		\end{split}
		\end{equation}
		Since the process comes down from infinity, for any $a>1$ and $n\ge1$,
		$\mbb{P}_\infty(T_{a^{n-1}}<\infty)=1.$
		By the strong Markov property and (\ref{p6.4f})
		\[\sum_{n=m}^\infty \mbb{P}_{\infty}\left(T^+_{a^{n+1}}\circ\theta_{T_{a^{n}}}<T_{a^{n-1}}\circ\theta_{T_{a^{n}}}\right)=\sum_{n=m}^\infty \mbb{P}_{a^n}\left(T^+_{a^{n+1}}<T_{a^{n-1}}\right)=\infty.\]
		Applying the  Borel-Cantelli lemma we have $\mbb{P}_\infty$-a.s. for infinitely many  $n$,
		$ T^+_{a^{n+1}}\circ\theta_{T_{a^{n}}}<T_{a^{n-1}}\circ\theta_{T_{a^{n}}}$,
		and $\limsup_{t\rightarrow 0+}\frac{X_t}{{\underline X}_t}\geq a$ $\mbb{P}_\infty$-a.s for all $a>1$. Thus, $\limsup_{t\rightarrow 0+}\frac{X_t}{{\underline X}_t}=\infty$ $\mbb{P}_\infty$-a.s.
	\end{proof}
	\subsection{Proofs of the main results}
	By  Lemma \ref{liminf_a} 
	we can now identify the speeds of coming down from infinity for different rate functions. We shall need the following result on functions satisfying $\mathbb{H}_1$ and $\mathbb{H}_2$. We refer the reader to Buldygin et al. \cite[Theorem 3.1 and Theorem 6.2]{MR1952816}, see also Djur\v{c}i\'{c} and Torga\v{s}ev \cite{MR1815787}.
	\begin{lemma}\label{inverse-equiv} Assume that $\varphi$ satisfies $\mathbb{H}_1$ and $\mathbb{H}_2$. For any nonnegative functions $u,v$ on $\mbb{R}_+$ such that $u(x)\underset{x\rightarrow \infty}{\sim} v(x)$ and $u(x)\underset{x\rightarrow \infty}{\longrightarrow} 0$, we have $\varphi^{-1}(u(x))\underset{x\rightarrow \infty}{\sim} \varphi^{-1}(v(x))$. If $\varphi(x)\underset{x\rightarrow \infty}{\sim} g(x)$ for some nonnegative decreasing function $g$ on $\mbb{R}_+$, then $\varphi^{-1}(t)\underset{ t\rightarrow 0+}{\sim} g^{-1}(t)$. 
	\end{lemma}
	\begin{proof}[Proof of Theorem \ref{regularCDI}]
		Under $\mathbb{H}_1$, Lemma \ref{equivmphi} and Proposition \ref{cvinproba} entail $m(b)\underset{b\rightarrow \infty}{\sim}\varphi(b)$ and the convergence in probability towards $1$ of $\left(T_b/m(b),b\geq 0\right)$.  Thus, since $\mathbb{H}_2$ entails $\underset{x\rightarrow \infty}{\liminf}\frac{m(x)}{m(hx)}>1$ for any $h>1$, by Lemma \ref{liminf_a}, in $\mathbb{P}_\infty$-probability
		$\frac{\underbar{X}_t}{m^{-1}(t)}\underset{t\rightarrow 0+}{\longrightarrow} 1$.
		By Proposition \ref{equivinf}, $\underbar{X}_t\underset{t\rightarrow 0+}{\sim} X_t$ almost-surely and Theorem \ref{regularCDI} follows since by Lemma \ref{inverse-equiv}, $m^{-1}(t)\underset{ t\rightarrow 0+}{\sim} \varphi^{-1}(t)$.
	\end{proof}
	\begin{proof}[Proof of Theorem \ref{regularvarR}]  By the assumption, $R$ is regularly varying at $\infty$ with index $\theta>1$ and
		Corollary \ref{stablelikepsiandcramer} ensures that  \eqref{cdicondition} is satisfied and $\mathbb{P}_\infty$-almost-surely, $\frac{T_b}{m(b)}\underset{b\rightarrow \infty}{\longrightarrow} 1$. By Lemma \ref{equivmphi}, $m(b)\underset{b\rightarrow \infty}{\sim} \varphi(b)$ and Lemma \ref{inverse-equiv} entails
		$\frac{m^{-1}(T_b)}{b}\underset{b\rightarrow \infty}{\longrightarrow} 1$ $\mathbb{P}_\infty$-a.s. Since for any $h>1$ and for any $t\in (T_{hb},T_{b})$,  we have $b\leq \underbar{X}_t\leq hb$. Then
		$ \frac{b}{m^{-1}(T_{hb})}\leq \frac{\underbar{X}_t}{m^{-1}(t)}\leq \frac{bh}{m^{-1}(T_b)}$.
		Therefore,  $\mathbb{P}_\infty$ almost-surely, $\underset{t\rightarrow 0+}{\liminf}\frac{\underbar{X}_t}{m^{-1}(t)}\geq \frac{1}{h}$ and $\underset{t\rightarrow 0+}{\limsup}\frac{\underbar{X}_t}{m^{-1}(t)}\leq h$. Since $h$ can be arbitrarily close to $1$, we have
		$\underset{t\rightarrow 0+}{\limsup}\frac{\underbar{X}_t}{m^{-1}(t)}=\underset{t\rightarrow 0+}{\liminf}\frac{\underbar{X}_t}{m^{-1}(t)}=1\quad \mathbb{P}_\infty\text{-a.s.}$
		By Proposition \ref{equivinf}, $\underbar{X}_t\underset{t\rightarrow 0+}{\sim} X_t$\,\, $\mathbb{P}_\infty\text{-a.s.} $ and
		since by Lemma \ref{inverse-equiv}, $m^{-1}(t)\underset{t\rightarrow 0+}{\sim} \varphi^{-1}(t)$, we have that
		$\frac{X_t}{\varphi^{-1}(t)}\underset{t\rightarrow 0+}{\longrightarrow} 1$ $\mathbb{P}_\infty$-a.s. Moreover, \cite[Theorem 1.5.12]{MR1015093} entails that $\varphi^{-1}$ is regularly varying at $0$ with index $-1/(\theta -1)$.
	\end{proof}
	\begin{proof}[Proof of Theorem \ref{stablecriticalrv}]
		Assume $\Psi(\lambda)\sim c\lambda^{\alpha}$ with $\alpha\in (1,2]$ and $c>0$ and $R$ regularly varying at $\infty$ with index $\theta>\alpha$.  We have seen in Corollary \ref{cvlawcriticalstablerv} that \eqref{cdicondition} is satisfied. The first statement $\limsup_{t\rightarrow 0+}\frac{X_t}{{\underline X}_t}=\infty$ is given by Proposition \ref{bigexcursionsabove}.
		For any $t>0$ and any $y\geq 0$,
		\begin{align*}
		\mathbb{P}_\infty\left(\frac{\underbar{X}_t}{m^{-1}(t)}\leq y\right)&=\mathbb{P}_\infty(T_{ym^{-1}(t)}\leq t)=\mathbb{P}_\infty\left(\frac{T_{ym^{-1}(t)}}{m(ym^{-1}(t))}\leq \frac{t}{m(ym^{-1}(t))}\right).
		\end{align*}
		By (\ref{equivmstable}), $m$ is a regularly varying function with index $\alpha-\theta<0$ at $\infty$ and
		$\frac{t}{m(ym^{-1}(t))}=\frac{m(m^{-1}(t))}{m(ym^{-1}(t))}\underset{t\rightarrow 0+}{\longrightarrow} y^{\theta-\alpha}$. By Corollary \ref{cvlawcriticalstablerv}, for any $y\geq 0$,
		$\mathbb{P}_\infty\left(\frac{\underbar{X}_t}{m^{-1}(t)}\leq y\right)\underset{t\rightarrow 0+}{\longrightarrow }\mathbb{P}\left(S_{\alpha,\theta}^{\frac{1}{\theta-\alpha}}\leq y\right)$.  By \cite[Theorem 1.5.12]{MR1015093}, $m^{-1}$ is regularly varying at $0$ with index $\frac{1}{\alpha-\theta}$.
	\end{proof}
	\begin{proof}[Proof of Proposition \ref{fastCDI}]
		Assume $R(x)=x^{\theta_1}\ell(x)e^{\theta_2 x}$ for $\theta_1\in \mathbb{R}$ and $\theta_2>0$, where $\ell$ is a slowly varying function at $\infty$. We have seen in  Corollary \ref{exponen_conve} that \eqref{cdicondition} is satisfied and the inequality \eqref{W1encadrement} yields that for some $\epsilon>0$ and large enough $x$,
		$\frac{m(x)}{m(xh)}\geq \frac{(1-\epsilon)^2}{1+\epsilon}\frac{R(hx)}{R(x)}$.
		Since for $h>1$, $\liminf_{x\rightarrow\infty}\frac{R(hx)}{R(x)}=\infty$,
		we have $\liminf_{x\rightarrow\infty}\frac{m(x)}{m(hx)}=\infty$
		and by Corollary \ref{exponen_conve} and  Lemma \ref{liminf_a}, in $\mathbb{P}_\infty$-probability
		$\frac{\underbar{X}_t}{m^{-1}(t)}\underset{t\rightarrow 0+}{\longrightarrow} 1$
		and Proposition \ref{equivinf} entails $\underbar{X}_t \underset{t\rightarrow 0+}{\sim} X_t$. It remains to show that, when $\gamma>0$, $m^{-1}(t)\underset{t\rightarrow 0+}{\sim} \varphi^{-1}(t).$   By Corollary \ref{exponen_conve},  $m(x)\underset{x\rightarrow \infty}{\sim} \frac{1}{\Psi(\theta_2) R(x)}$ and  $\varphi(x)\underset{x\rightarrow \infty}{\sim} \frac{1}{\theta_2 \gamma R(x)}$ (for the latter equivalence, consider $\Psi(\lambda)=\gamma \lambda$ and note that $W_1=\varphi$). Therefore, $\frac{m(x)}{\varphi(x)}\underset{x\rightarrow \infty}{\longrightarrow} \frac{\gamma \theta_2}{\Psi(\theta_2)}.$  For any $\lambda>\frac{\Psi(\theta_2)}{\gamma \theta_2}$ and $\mu>\frac{\gamma \theta_2}{\Psi(\theta_2)}$, for large enough $x$, $\mu \varphi(x)\geq m(x)\geq \frac{1}{\lambda}\varphi(x)$
		and  for $t$ small enough
		$\varphi^{-1}(t/\mu)\leq m^{-1}(t)\leq \varphi^{-1}(\lambda t)$.
		By Elez and Djur\v{c}i\'{c} \cite[Theorem 1.1-(d)]{MR3018036}, since $\varphi(x)/\varphi(cx)\underset{x\rightarrow \infty}{\longrightarrow} \infty$ for any $c>1$,   $\varphi^{-1}$ is slowly varying at $0$, and  $\frac{m^{-1}(t)}{\varphi^{-1}(t)}\underset{t\rightarrow 0+}{\longrightarrow} 1.$
	\end{proof}
	\textbf{Acknowledgements:} C.F's research is partially  supported by the French National Research Agency (ANR): ANR GRAAL (ANR-14-CE25-0014) and by LABEX MME-DII (ANR11-LBX-0023-01). X.Z.'s research is supported by Natural Sciences and Engineering Research Council of Canada (RGPIN-2016-06704) and National Natural Science Foundation of China (No.\ 11731012 and No.\ 11771018). P.L.'s research is supported by Natural Sciences and Engineering Research Council of Canada (RGPIN-2012-07750 and RGPIN-2016-06704) and National Natural Science Foundation of China (No.\ 11901570 and No.\ 11771046). The authors are grateful to Donald Dawson for his generous support and encouragement. They also thank an anonymous referee for very detailed comments.

	\bigskip
	
	\bigskip
	
	\noindent\textbf{\Large Appendix A: Proof of Proposition 2.1} 
	
	\bigskip
	
	Recall the statement of Proposition 2.1. Since $R$ is continuous and positive,  for any $x>0$ and any $t<\tau_0^{-}$, $\mathbb{P}_x$-almost-surely
	$R(Z_s)\geq \inf_{u\leq t} R(Z_u)=\inf_{z\in [\underset{s\leq t}\inf Z_s, \underset{s\leq t}\sup Z_s]}R(z)>0$ for any $s\leq t $. This provides that $\eta(t)<\infty$ a.s. and entails that the process $X$ is well-defined. Time-changing a strong Markov process with c\`adl\`ag paths by the inverse of an additive functional gives another c\`adl\`ag strong Markov process, see e.g. Volkonskii \cite{MR0100919}. In particular, property (ii) is satisfied and it only remains to verify (i) on $[0,\infty)$. The argument follows closely that of Lamperti \cite{MR0230370} and is based on a continuity result due to Whitt \cite[Theorem 3.1]{MR561155} and Caballero  et al. \cite[Section 3]{MR2592395}. Denote by $\mathcal{D}$ the space of c\`adl\`ag paths from $[0,\infty)$ to $[0,\infty]$, with $0$ and $\infty$ as traps. For any $x,y\in [0,\infty]$, set  $\rho(x,y)=|e^{-x}-e^{-y}|$ and $\rho_\infty^{\mathcal{D}}(f,g)=\sup_{s\in  [0,\infty)}\rho(f(s),g(s))$.
	Let $\Lambda_\infty$ be the set of increasing homeomorphisms of $[0,\infty)$ into itself and $d_\infty$ the metric on $\mathcal{D}$
	\[d_\infty(f,g):=\underset{\lambda \in \Lambda_\infty}{\inf} \! \! \! \left(\rho_\infty^{\mathcal{D}}(f,g\circ \lambda)\vee ||\lambda-\mathrm{Id}||\right),\]
	where $||\cdot||$ denotes the uniform norm on $[0, \infty)$. Consider the canonical L\'evy process $Z^0$, with Laplace exponent $\Psi$, started from $0$ and let $Z^{x}:=x+Z^{0}$ be the L\'evy process  started from $x\in [0,\infty)$. Set $\tau_0^{x}:=\inf\{t\geq 0: Z_t^x\leq 0\}$ and $(X^x_{t},t\geq 0)$ the nonlinear CSBP started from $x$ defined by $X_t^x=Z^{x}_{\eta^{-1}(t)\wedge \tau^x_0}$ for all $t\geq 0$. Same arguments as in Caballero et al. \cite[Proposition 5]{MR2592395} readily apply and entail that the time-change transformation in $\mathcal{D}$ mapping sample paths of $Z^x_{\cdot\wedge \tau^{x}_0}$ to sample paths of $X^x$ is continuous with respect to the distance $d_\infty$. 
	One can  check, see the proof of Theorem 1.2 in Li \cite{2016arXiv160909593L}, that almost-surely
	\begin{equation}\label{continuityfortheLP}
	\underset{x\rightarrow y}{\lim}\ d_\infty(Z^{x}_{\cdot\wedge \tau_0^{x}},Z^{y}_{\cdot\wedge\tau_0^{y}})=\underset{y\rightarrow x}{\lim}\ d_\infty(Z^{y}_{\cdot\wedge \tau_0^{y}},Z^{x}_{\cdot\wedge\tau_0^{x}})=0.
	\end{equation}
	Therefore, by the continuity, \eqref{continuityfortheLP} entails
	\begin{equation}\label{continuityfortheNCSBP}
	\underset{x\rightarrow y}{\lim}\ d_\infty(X^{x}_{\cdot},X^{y}_{\cdot})=\underset{y\rightarrow x}{\lim}\ d_\infty(X^{y}_{\cdot},X^{x}_{\cdot})=0 \text{ a.s.}
	\end{equation}
	According to Billingsley's book \cite[Lemma 1, Section 16] {MR1700749}, convergence with $d_\infty$ implies convergence in the usual Skorokhod distance. Recall also that for any $t\in (0,\infty)$, the projection $\pi_t: \mathrm{x}\in \mathcal{D}\mapsto \mathrm{x}(t)$ is continuous on the set $\{\mathrm{x}\in \mathcal{D}; \mathrm{x}(t)=\mathrm{x}(t-)\}$ for the Skorokhod distance and thus for $d_\infty$. Moreover, for any $t$ and any $y$, $\mathbb{P}(X^{y}_{t-}=X^{y}_{t})=1$ and \eqref{continuityfortheNCSBP} ensures that for any continuous bounded function $f$, $f(X^{x}_t)\underset{x\rightarrow y}{\longrightarrow} f(X^{y}_t)$ almost-surely. Lebesgue's theorem finally entails $P_tf(x)\underset{x\rightarrow y}{\longrightarrow} P_tf(y).$
	
\bigskip

\bigskip

\noindent\textbf{\Large Appendix B: Proof of Proposition 2.7} 

\bigskip
	Recall  Theorem 2.7. Consider a spectrally positive L\'evy process $Z$ and a locally bounded nonnegative function $\omega$ on $(0, \infty)$.  Denote by $\hat{\mathbb{P}}_x$ the law of $\hat{Z}:=-Z$ started at $x$. Set $\hat{\tau}^+_y:=\inf\{t\ge0: \hat{Z}_t> y\}$ and $\hat{\tau}^-_y:=\inf\{t\ge0: \hat{Z}_t< y\}$.
	Let $W^{(\omega)}_n(x)$ satisfy
	$W^{(\omega)}_0(x)=1$, $W^{(\omega)}_{n+1}(x)=\int_{x}^{\infty}W(z-x)\omega(z)W^{(\omega)}_n(z)\ddr z$ for $ x\geq 0, n\ge0$.
	Given $b\geq 0$, we assume
	$\sum^\infty_{n=0}  W^{(\omega)}_{n}(b)<\infty.$
	We show in the sequel that for all $x\geq b$,
	\[ \mbb{E}_x\left[\exp\left(-\int_0^{\tau_b^-} \omega(Z_s)ds\right); \tau^-_b<\infty  \right]=\frac{\sum^\infty_{n=0} W^{(\omega)}_{n}(x)}{\sum^\infty_{n=0} W^{(\omega)}_{n}(b)}.\]
	Recall that $\tau^{-}_b$ under $\mathbb{P}_x$ has the same law as $\hat{\tau}^+_{-b}$ under $\hat{\mathbb{P}}_{-x}$. For $a>0$ and $y<0$ let $\omega_a(y):=\omega(-y)1_{y\geq -a} $. Applying Corollary 2.2 of \cite{MR3849809}  we have for $-\gamma<-a<-z_0<-b<0$ and with $\hat{Z}_0=-Z_0=-z_0$,
	\[\mbb{\hat E}_{ -z_0} \left[ \exp\left(-\int_0^{\hat{\tau}^{+}_{-b}} \omega_a(\hat{Z}_s)ds\right); \, \hat{\tau}^{+}_{-b}<\hat{\tau}^-_{-\gamma} \right]=
	\frac{W^{(\omega_a)}(-z_0,-\gamma)}{W^{(\omega_a)}(-b,-\gamma)}, \]
	where for any $0>x>y$, $W^{(\omega_a)}(x,y)$ satisfies equation
	\begin{equation}\label{integral_equ}
	W^{(\omega_a)}(x,y)=W(x-y)+\int_y^x W(x-z)\omega_a(z)W^{(\omega_a)}(z,y)\ddr z.
	\end{equation}
	Similarly to the proof of  Li and Palmowski \cite[Theorem 2.5]{MR3849809}, we can show that
	$H^{(\omega_a)}(x):=\lim_{\gamma\rightarrow\infty}\frac{W^{(\omega_a)}(x,-\gamma)}{W(x+\gamma)}$ exists for all $x<0$ and
	\[\mbb{\hat E}_{-z_0} \left[ \exp\left(-\int_0^{\hat{\tau}^{+}_{-b}} \omega_a(\hat{Z}_s)ds\right); \, \hat{\tau}^{+}_{-b}<\infty \right]=\frac{H^{(\omega_a)}(-z_0)}{H^{(\omega_a)}(-b)}. \]
	In addition, dividing both sides of equation (\ref{integral_equ}) by $W(x-y)$ and taking limits as $y\rightarrow -\infty$, we see that the function $H^{(\omega_a)}(x), x<0$, is the unique solution to the equation
	\[H^{(\omega_a)}(x)=1+\int_{-a}^x W(x-z)\omega(-z)H^{(\omega_a)}(z)\ddr z. \]
	For $x>0$ put $W^{(\omega,a)}(x):=H^{(\omega_a)}(-x)$. Then $W^{(\omega, a)}(x)$ solves equation
	\[W^{(\omega, a)}(x)=1+\int^{a}_x  W(z-x)\omega(z)W^{(\omega, a)}(z)\ddr z. \]
	It follows that
	$W^{(\omega,a)}(x)=\sum^\infty_{n=0} W^{(\omega,a)}_n(x)$
	where for $0<x<a$,
	$W^{(\omega,a)}_0(x)=1$, $W^{(\omega, a)}_{n+1}(x)=\int_{x}^a W(z-x)\omega(z)W^{(\omega,a)}_n(z)\ddr z$.
	By induction on $n$ we can show that $W^{(\omega,a)}_n(x)$ increases in $a$ and decreases in $x$. Write
	$W^{(\omega)}_n(x):=\lim_{a\rightarrow\infty} W^{(\omega, a)}_n(x), \quad x>0$.
	Then $W^{(\omega)}_n$ satisfies (2.5) and for $x\geq b$,
	$\sum^\infty_{n=0} W^{(\omega)}_n(x)\leq \sum^\infty_{n=0} W^{(\omega)}_n(b)<\infty.$
	Finally, by monotone convergence
	\begin{equation}
	\begin{split}
	&\mbb{E}_x\left[\exp\left(- \int_0^{\tau_b^{-}} \omega(Z_s)ds\right); \tau_b^{-}<\infty  \right]\\
	&=\hat{\mbb{E}}_{-x}\left[\exp\left(- \int_0^{\hat{\tau}^{+}_{-b}} \omega(-\hat{Z}_s)ds\right); \hat{\tau}^{+}_{-b}<\infty  \right]\\
	&=\lim_{a\rightarrow\infty}\hat{\mbb{E}}_{-x}\left[\exp\left(- \int_0^{\hat{\tau}^{+}_{-b}} \omega_a(\hat{Z}_s)ds\right); \hat{\tau}^{+}_{-b}<\infty  \right]\\
	&=\lim_{a\rightarrow\infty}\frac{W^{(\omega,a)}(x)}{W^{(\omega,a)}(b)}
	=\lim_{a\rightarrow\infty}\frac{\sum^\infty_{n=0} W^{(\omega,a)}_{n}(x)}{\sum^\infty_{n=0} W^{(\omega,a)}_{n}(b)}
	=\frac{\sum^\infty_{n=0} W^{(\omega)}_{n}(x)}{\sum^\infty_{n=0} W^{(\omega)}_{n}(b)}.\\
	\end{split}
	\end{equation}
	
	\begin{remark}
		Observe from the proof of Theorem 2.7 that the result holds without the assumption of $\gamma\geq 0$.
	\end{remark}


	\bibliographystyle{plain}

\end{document}